\documentclass[12pt]{amsart}
\usepackage{custom_dirk}
\usepackage[backend=biber]{biblatex} 
\usepackage{todonotes}
\usepackage{tcolorbox}
\usepackage[pdfencoding=auto]{hyperref}
\bibliography{references.bib}

\DeclareMathOperator{\Br}{Br}

\DeclareMathOperator{\QCoh}{QCoh} 
\DeclareMathOperator{\Sh}{Sh}
\usepackage{lipsum}

\newcommand\blfootnote[1]{%
  \begingroup
  \renewcommand\thefootnote{}\footnote{#1}%
  \addtocounter{footnote}{-1}%
  \endgroup
}

\begin{document}
\title[A virtual $\mathrm{PGL}_r$--$\mathrm{SL}_r$ correspondence]{A virtual $\mathrm{PGL}_r$--$\mathrm{SL}_r$ correspondence for projective surfaces}
\author[Van Bree, Gholampour, Jiang, Kool]{D.~van Bree, A.~Gholampour, Y.~Jiang, and M.~Kool}
\maketitle

\vspace{-1cm}
\begin{abstract}
For a smooth projective surface $X$ satisfying $H_1(X,\mathbb{Z}) = 0$ and $w \in H^2(X,\mu_r)$, we study deformation invariants of the pair $(X,w)$. Choosing a Brauer-Severi variety $Y$ (or, equivalently, Azumaya algebra $\mathcal{A}$) over $X$ with Stiefel-Whitney class $w$, the invariants are defined as virtual intersection numbers on suitable moduli spaces of stable twisted sheaves on $Y$ constructed by Yoshioka (or, equivalently, moduli spaces of $\mathcal{A}$-modules of Hoffmann-Stuhler). 

We show that the invariants do not depend on the choice of $Y$. Using a result of de Jong, we observe that they are deformation invariants of the pair $(X,w)$. For surfaces with $h^{2,0}(X) > 0$, we show that the invariants can often be expressed as virtual intersection numbers on Gieseker-Maruyama-Simpson moduli spaces of stable sheaves on $X$. This can be seen as a $\mathrm{PGL}_r$--$\mathrm{SL}_r$ correspondence.

As an application, we express $\mathrm{SU}(r) / \mu_r$ Vafa-Witten invariants of $X$ in terms of $\mathrm{SU}(r)$ Vafa-Witten invariants of $X$.  We also show how formulae from Donaldson theory can be used to obtain upper bounds for the minimal second Chern class of Azumaya algebras on $X$ with given division algebra at the generic point.
\end{abstract}

\section{Introduction} \label{sec:introduction}

\subsection{\texorpdfstring{$\mathrm{SL}_r$}{SLr} invariants} \label{sec:introSL} 

Let $(X,H)$ be a smooth polarized surface over the complex numbers. Suppose $H_1(X,\mathbb{Z}) = 0$ and take $r \in \mathbb{Z}_{>0}$, $c_1 \in H^2(X,\mathbb{Z})$, $c_2 \in H^4(X,\mathbb{Z}) \cong \mathbb{Z}$.\blfootnote{Keywords: moduli of twisted sheaves, virtual intersection numbers. MSC classes: 14D20, 14D21, 14F22, 14J60, 14J80.} We denote by $M:=M_X^H(r,c_1,c_2)$ the Gieseker-Maruyama-Simpson moduli space of rank $r$ Gieseker $H$-stable sheaves $F$ on $X$ satisfying $c_1(F) = c_1$ and $c_2(F) = F$ \cite{HL}. Then $M$ is a quasi-projective scheme and it has a natural ``compactification'' by adding strictly semistable sheaves. Sometimes $M$ is itself projective, e.g., when $\gcd(r,c_1H) = 1$, in which case Gieseker stability and $\mu$-stability coincide, and there are no rank $r$ strictly Gieseker $H$-semistable sheaves on $X$ with Chern classes $c_1,c_2$. We view $M$ as a partial compactification of moduli spaces of holomorphic vector bundles, i.e., holomorphic principal $\mathrm{GL}_r$ bundles.\footnote{Since we assume $H_1(X,\mathbb{Z}) = 0$ fixing $c_1$ is equivalent to fixing the determinant, so for $c_1=0$ we are considering holomorphic principal $\mathrm{SL}_r$ bundles.} 

The moduli space $M$ is virtually smooth --- it has a perfect obstruction theory, studied by Mochizuki \cite{Moc}, with virtual tangent bundle 
\[
T_{M}^{\mathrm{vir}} = R \mathcal{H}{\it{om}}_{\pi_M} (\mathcal{E},\mathcal{E})_0[1]
\]
where $\pi_M : X \times M \to M$ denotes the projection, $R \mathcal{H}{\it{om}}_{\pi_M} = R\pi_{M*} \circ R \mathcal{H}{\it{om}}$, $(\cdot)_0$ denotes trace-free part, and $\mathcal{E}$ is a universal sheaf on $X \times M$. In general, a universal sheaf $\mathcal{E}$ only exists \'etale locally on $X \times M$, but the virtual tangent bundle exists globally by \cite[Thm.~2.2.4]{Caldararu2000} (see also \cite[Sect.~10.2]{HL}). 
We assume $\gcd(r,c_1H)=1$.
Then $M$ is projective and, by work of Behrend-Fantechi \cite{BF} and Li-Tian \cite{LT}, there exists a virtual fundamental class
\[
[M]^{\mathrm{vir}} \in A_{\mathrm{vd}}(M), \quad \mathrm{vd} := \mathrm{vd}(r,c_1,c_2) = \rk(T_M^{\mathrm{vir}}) =  2rc_2 - (r-1)c_1^2 - (r^2-1)\chi(\mathcal{O}_X),
\]
where $A_*(M)$ denotes the Chow group of $M$. 

Intersection numbers obtained by capping with the virtual fundamental class play a key role in enumerative geometry, gauge theory, and physics. Examples of such intersection numbers are virtual Euler characteristics, $\chi_y$-genera, elliptic genera, cobordism classes, Donaldson invariants, Segre numbers, and Verlinde numbers. In the cases $HK_X < 0$ or $K_X \cong \mathcal{O}_X$, $M$ is smooth of expected dimension and these numbers have a long history going back (at least) to the 1990s. 
In this paper, we focus on surfaces with a non-zero holomorphic 2-form, i.e.~$h^{2,0}(X) > 0$, in which case $M$ is typically singular and may not have expected dimension (see e.g.~\cite{MS} for examples). For a (partial!) survey on this rich subject, we refer to \cite{GK5} and references therein. One key feature of these ``virtual intersection numbers'' is that they are invariant under deformations of $X$.

Perhaps the most interesting case is the generating function of virtual Euler characteristics 
\[
\mathsf{Z}^{\mathrm{SL}_r,\mathsf{Eu}}_{(X,H),c_1}(q) =  \sum_{c_2} q^{\frac{\mathrm{vd}(r,c_1,c_2)}{2r}} \int_{[M_X^H(r,c_1,c_2)]^{\mathrm{vir}}} c(T_{M_X^H(r,c_1,c_2)}^{\mathrm{vir}}),
\]
where $c(\cdot)$ denotes the total Chern class, which plays a central role in Vafa-Witten theory \cite{TT}. In the ground-breaking work of Vafa-Witten \cite{VW} on $S$-duality and $N=4$ supersymmetric Yang-Mills theory on the 4-manifold underlying $X$, such generating functions are predicted to be Fourier expansions of meromorphic functions on the upper half plane with beautiful modular properties. Notably, under the $S$-duality transformation, a certain generating function associated to gauge group $\mathrm{SU}(r)$ and its Langlands dual $\mathrm{PSU}(r) = \mathrm{SU}(r) / \mu_r$ are related. We denote by $\mu_r$ the multiplicative cyclic group of order $r$.

In this paper, we consider essentially arbitrary intersection numbers on Gieseker-Maruyama-Simpson moduli spaces obtained as polynomial expressions in ``descendent insertions'' as defined in Section \ref{sec:genfun}. Besides virtual Euler characteristics this includes, e.g., the Segre and Verlinde numbers studied in \cite{MOP, GK4, Yua, GM}. For a choice of formal insertions $\mathsf{P}$, we denote the corresponding generating function by $\mathsf{Z}^{\mathrm{SL}_r,\mathsf{P}}_{(X,H),c_1}(q)$. There are several powerful tools for the calculation of virtual intersection numbers, notably Mochizuki's formula \cite{Moc} and the new vertex algebra wall-crossing technology developed by Joyce and collaborators \cite{GJT, Joy}.

\subsection{\texorpdfstring{$\mathrm{PGL}_r$}{PGLr} invariants} 

In this paper, we are interested in compactifications of moduli spaces of holomorphic $\mathbb{P}^{r-1}$-bundles, i.e., holomorphic principal $\mathrm{PGL}_r$ bundles with $r>1$. 
The correct approach for dealing with these objects within algebraic geometry is by using moduli spaces of \emph{twisted sheaves}. Instead of first Chern class $c_1$, we now fix a Stiefel-Whitney class\footnote{In physics parlance, $w$ is called the 't Hooft flux \cite{VW}.} $w \in H^2(X,\mu_r)$, where we view the multiplicative group $\mu_r$ of $r$th roots of unity as a constructible sheaf on $X$ in the \'etale topology. In fact, $H^2(X,\mu_r)$ is isomorphic to the singular cohomology group $H^2(X,\mathbb{Z} / r\mathbb{Z})$, where $X$ is endowed with its complex analytic topology \cite[Thm.~21.1]{MilneLec}. The inclusion $\mu_r \leq \mathbb{G}_m$ induces a map
\[
o : H^2(X,\mu_r) \to H^2(X,\mathbb{G}_m),
\]
where $H^2(X,\mathbb{G}_m)$ is isomorphic to the Brauer group $\mathrm{Br}(X)$. For $\alpha:=o(w)$, there are many models for twisted sheaves on $X$:
\begin{itemize}
\item \textbf{C\u{a}ld\u{a}raru.} \cite{Caldararu2000} Perhaps the most intuitive approach is to represent $\alpha \in H^2(X,\mathbb{G}_m)$ by a \v{C}ech 2-cocycle $\{\alpha_{ijk} \in H^0(U_{ijk}, \mathbb{G}_m)\}$, where $\{U_i \to X\}$ is an \'etale cover and we write $U_{ij} = U_i \times_X U_j$, $U_{ijk} = U_i \times_X U_j \times_X U_k$. Then an $\alpha$-twisted sheaf consists of a collection of sheaves $\{F_i \in \mathrm{Coh}(U_i) \}$ together with isomorphism $\{\phi : F_i |_{U_{ij}} \to F_j|_{U_{ij}}\}$ satisfying $\phi_{ii} = \mathrm{id}$, $\phi_{ji} = \phi_{ij}^{-1}$, and $\phi_{ij} \circ \phi_{jk} \circ \phi_{ki} = \alpha_{ijk} \cdot \mathrm{id}$ for all $i,j,k$.
\item \textbf{Yoshioka.} \cite{Yoshioka2006} Let $[\pi : Y \to X] \in H^1(X,\mathrm{PGL}_r)$ be a degree $r$ Brauer-Severi variety, i.e.~\'etale $\mathbb{P}^{r-1}$ fibre bundle, with Stiefel-Whitney class $w(Y) = w \in H^2(X,\mu_r)$. 
It is a consequence of the period-index theorem, proved by A.J.~de Jong \cite{Jong2004}, \cite[Cor.~4.2.2.4]{Lieblich2008}, that such a Brauer-Severi variety exists.
Yoshioka defines the notion of $Y$-sheaves. These are essentially pull-backs of $\alpha$-twisted sheaves from $X$ to $Y$ tensored with the $-\pi^* \alpha$-twisted line bundle $\mathcal{O}_Y(1)$, which removes the twist.
\item \textbf{Lieblich.} \cite{Lieblich_2007} Let $\mathcal{G} \to X$ be the $\mu_r$-gerbe associated to $w \in H^2(X, \mu_r)$. Then Lieblich introduces twisted sheaves on $\mathcal{G}$. Roughly speaking, coherent sheaves on $\mathcal{G}$ decompose with respect to the character group of $\mu_r$ and the twisted sheaves are the weight 1 eigensheaves. 
\item \textbf{Hoffmann-Stuhler.} \cite{HoffmannStuhler2005} One can also view elements of $H^1(X,\mathrm{PGL}_r)$ as isomorphism classes of degree $r$ Azumaya algebras on $X$. Indeed, there exists a non-trivial extension $0 \to \mathcal{O}_Y \to G \to T_{Y/X} \to 0$ (unique up to scaling), where $T_{Y/X}$ denotes the relative tangent bundle. Then $\mathcal{A} = \pi_*(\mathcal{E}{\it{nd}}(G^\vee))$ is the degree $r$ Azumaya algebra corresponding to $Y$ (Lemma \ref{lem:sheaf-G}, Proposition \ref{prp:azu-bs-equiv}). Suppose $\mathcal{A}$ has Stiefel-Whitney class $w \in H^2(X,\mu_r)$.
Then Hoffmann-Stuhler introduce moduli spaces of (left) $\mathcal{A}$-modules which are generically simple. These modules provide another model for twisted sheaves. Azumaya algebras are generalizations of central simple algebras, which in turn are generalizations of division algebras. The latter have a long history dating back to Hamilton. 
\end{itemize}
We survey these models for twisted sheaves and their equivalences in Section \ref{sec:twisted-sheaves-and-models}. We will mostly work with the moduli spaces of Yoshioka and Hoffmann-Stuhler, and occasionally the one of Lieblich (Section \ref{subsec:moduli}).

Let $Y \to X$ be a degree $r$ Brauer-Severi variety with Stiefel-Whitney class $w:=w(Y) \in H^2(X,\mu_r)$.
Then we consider the generating function of virtual Euler characteristics\footnote{As it stands, the second Chern class $c_2$ is rational. It can be made integral by ``twisting by the $B$-field'' \cite{HuybrechtsStellari2005} (Proposition \ref{prop:integrality}).}
\[
\mathsf{Z}^{\mathrm{PGL}_r,\mathsf{Eu}}_{(X,H),w}(q) =  \sum_{c_2} q^{\frac{\mathrm{vd}(r,0,c_2)}{2r}} \int_{[M_Y^H(r,0,c_2)]^{\mathrm{vir}}} c(T_{M_Y^H(r,0,c_2)}^{\mathrm{vir}}),
\]
where we assume all of the moduli spaces of $H$-stable $Y$-sheaves $M_Y^H(r,0,c_2)$ (which we recall in Section \ref{sec:moduli-of-twisted-sheaves}) are projective.
For instance, this is the case when the Brauer class $o(w) \in H^2(X,\mathbb{G}_m)$ has order $r$. Indeed, then there are no $Y$-sheaves $F$ of rank $0 < \rk(F) < r$, thus stability is automatic and there are no rank $r$ strictly semistable $Y$-sheaves (Remark \ref{rem:stabauto}). In this case, the generating function does not depend on the choice of polarization $H$ and we write 
\[
\mathsf{Z}^{\mathrm{PGL}_r,\mathsf{Eu}}_{X,w}(q) = \mathsf{Z}^{\mathrm{PGL}_r,\mathsf{Eu}}_{(X,H),w}(q).
\]
We show in Proposition \ref{prop:indep} that the generating function does not depend on the choice of degree $r$ Brauer-Severi variety $Y \to X$. 

As in the $\mathrm{SL}_r$ case, we also consider arbitrary polynomial expressions in descendent insertions on moduli of twisted sheaves. For any choice of formal insertions $\mathsf{P}$, we denote the corresponding generating function by $\mathsf{Z}^{\mathrm{PGL}_r,\mathsf{P}}_{(X,H),w}(q)$ (see Section \ref{sec:genfun} for the precise definition).

The third-named author first introduced the use of twisted sheaves to Vafa-Witten theory in \cite{Jiang}. This was used by the third- and fourth-named authors in \cite{JiangKool2021} to introduce the $\mathrm{PSU}(r)$ Vafa-Witten partition function when $r$ is prime. Denoting $\epsilon_r := \exp(2\pi \sqrt{-1} / r)$, it has the following form
\begin{align}
\begin{split} \label{eqn:PSUVWgenfun}
&\mathsf{VW}^{\mathrm{PSU}(r)}_{X,c_1}(q) = \sum_{w \in H^2(X,\mu_r)} \epsilon_r^{c_1 w} \, \mathsf{VW}_{X,w}(q), \\
&\textrm{where} \quad  \mathsf{VW}_{X,w}(q) = \mathsf{Z}^{\mathrm{PGL}_r,\mathsf{Eu}}_{X,w}(q), \quad \mathrm{if \ } 0 \neq o(w) \in \mathrm{Br}(X).
\end{split}
\end{align}

\subsection{Main results} 

The construction of the $\mathrm{PGL}_r$ generating function on $X$ depends on a choice of a degree $r$ Brauer-Severi variety $Y \to X$ with $w(Y)=w$. Since $Y$ may be obstructed when deforming $X$, deformation invariance of the generating function is not immediate. However, by a result of de Jong \cite{Jong2004}, when $Y$ is obstructed one can always apply an elementary transformation after which it becomes unobstructed. We recall this result in Theorem \ref{thm:extending-azu-algs}. This will lead to the deformation invariance, which we will now describe.

Let $f : \mathcal{X} \to B$ be a smooth projective morphism of relative dimension 2 with connected fibres over a smooth connected variety $B$. Suppose that one fibre (and hence all fibres) $\mathcal{X}_b$ satisfies $H_1(\mathcal{X}_b,\mathbb{Z}) = 0$. We fix a family of polarizations $\mathcal{H}$ on $\mathcal{X}$ and we consider $\mu_r$ as a constructible sheaf in the \'etale topology on $\mathcal{X}$. Then $R^2 f_* \mu_r$ is a constructible sheaf and 
\[
(R^2 f_* \mu_r)_b \cong H^2(\mathcal{X}_b, \mu_r)
\]
for all closed points $b \in B$ by the proper base change theorem \cite[Thm.~17.7]{MilneLec}. We fix a section 
\[
\widetilde{w} \in H^0(B,R^2 f_* \mu_r).
\]

In this paper, we only want to consider the case where there are no strictly semistable objects anywhere in our family. This can be achieved by the following two assumptions:
\begin{itemize}
\item $r$ is prime, and
\item $\gcd(r,\widetilde{w}_b \mathcal{H}_b) = 1$ for some (and hence all) closed points $b \in B$.
\end{itemize}
Since $r$ is prime and $\widetilde{w}_b \in H^2(\mathcal{X}_b,\mu_r)$ for any closed point $b \in B$, the order of $o(\widetilde{w}_b) \in H^2(\mathcal{X}_b,\mathbb{G}_m)$ is either 1 or $r$. Then for any degree $r$ Brauer-Severi variety $Y$ over $\mathcal{X}_b$ and any $c_2$, the moduli space $M_Y^{\mathcal{H}_b}(r,0,c_2)$ is projective (Lemma \ref{lem:moduliprojinfam}). Roughly speaking, in the case of trivial Brauer class, we are in the untwisted setting and the second condition rules out strictly semistable objects (Proposition \ref{prop:untwist}), whereas in the case of non-trivial Brauer class, stability is automatic for all rank $r$ torsion free twisted sheaves, so in particular there are no strictly semistables (Remark \ref{rem:stabauto}). In order to deal with strictly semistable objects, one should work with a notion of twisted Joyce-Song-Mochizuki pairs \cite{JS, Joy, Moc}. 

\begin{theorem} \label{thm:main1}
Let $\widetilde{w} \in H^0(B,R^2 f_* \mu_r)$ be a section. Suppose $r$ is prime and $\mathrm{gcd}(r,\widetilde{w}_b \mathcal{H}_b)=1$ for some (and hence all) closed points $b \in B$. Then $\mathsf{Z}_{(\mathcal{X}_{b},\mathcal{H}_{b}),\widetilde{w}_{b}}^{\mathrm{PGL}_r, \mathsf{P}}(q)$ is independent of the closed point $b \in B$.
\end{theorem}

This leads us to the following $\mathrm{SL}_r$--$\mathrm{PGL}_r$ correspondence.
\begin{theorem} \label{thm:main2}
Let $\widetilde{w} \in H^0(B,R^2 f_* \mu_r)$ be a section. Suppose $r$ is prime and $\mathrm{gcd}(r,\widetilde{w}_b \mathcal{H}_b)=1$ for some (and hence all) closed points $b \in B$. Suppose for some closed point $0 \in B$, there exists a class $\beta \in H^{1,1}(\mathcal{X}_0)$ such that the following composition is surjective
\[
T_B|_{0} \stackrel{\mathrm{KS}_0}{\longrightarrow} H^1(\mathcal{X}_0, T_{\mathcal{X}_0}) \stackrel{\cup \beta}{\longrightarrow} H^{2}(\mathcal{X}_0,\mathcal{O}_{\mathcal{X}_0}),
\]
where the first arrow is the Kodaira-Spencer map and the second is cupping with $\beta$ followed by contraction. Then any complex analytic simply connected neighbourhood $U$ of $0$ contains a closed point $b \in U$ such that $\widetilde{w}_{b} \in H^2(\mathcal{X}_{b},\mu_r)$ has trivial Brauer class and
\[
\mathsf{Z}_{(\mathcal{X}_{0},\mathcal{H}_{0}),\widetilde{w}_{0}}^{\mathrm{PGL}_r, \mathsf{P}}(q) = \mathsf{Z}_{(\mathcal{X}_{b},\mathcal{H}_{b}),c_1}^{\mathrm{SL}_r, \mathsf{P}}(q),
\]
where $c_1 \in H^2(\mathcal{X}_{b},\mathbb{Z})$ is any (necessarily algebraic) representative of $\widetilde{w}_{b}$.
\end{theorem}

The condition of the theorem is mild --- it says that there exists \emph{at least one} $\beta \in H^{1,1}(\mathcal{X}_0)$ for which the Noether-Lefschetz locus is smooth of expected codimension \cite{VoiHL}. 
The same assumption is used by Green to show that the Hodge locus is dense \cite{VoiBookII}.
When $\beta$ is an effective algebraic class, this condition appears in the work of Thomas and the fourth-named author on reduced Gromov-Witten/stable pairs theory, and the enumerative geometry of curves in the linear system $|\beta|$ \cite{KT1, KT2}. 

We discuss examples for which the condition of the theorem is satisfied in Section \ref{sec:mainresult}. E.g.~when $X \subset \mathbb{P}^3$ is a smooth surface of degree $d \geq 4$, we consider $$\mathcal{X} \to B \subset |\mathcal{O}_{\mathbb{P}^3}(d)|$$ the smooth family of smooth degree $d$ surfaces. Then the locus of points $b \in B$ for which there exists a $\beta$ satisfying the condition of Theorem \ref{thm:main2} is dense by a result of Kim \cite{Kim_1991}. Hence, we can first use Theorem \ref{thm:main1} to deform to a point in $B$ where the condition of Theorem \ref{thm:main2} is satisfied, and then apply Theorem \ref{thm:main2} to express the $\mathrm{PGL}_r$ generating function for $X$ in terms of the $\mathrm{SL}_r$ generating function.

When $o(\widetilde{w}_0) \in \mathrm{Br}(\mathcal{X}_0)$ has order $r$ (so in particular $h^{2,0}(\mathcal{X}_0) > 0$), we note that $\mathsf{Z}_{(\mathcal{X}_{0},\mathcal{H}_{0}),\widetilde{w}_{0}}^{\mathrm{PGL}_r, \mathsf{P}}(q) = \mathsf{Z}_{\mathcal{X}_{0},\widetilde{w}_{0}}^{\mathrm{PGL}_r, \mathsf{P}}(q)$ does not depend on the polarization. In particular, it follows that $\mathsf{Z}_{(\mathcal{X}_{b},\mathcal{H}_{b}),c_1}^{\mathrm{SL}_r, \mathsf{P}}(q)$ does not depend on the polarization. It also follows that $\mathsf{Z}_{(\mathcal{X}_{b},\mathcal{H}_{b}),c_1}^{\mathrm{SL}_r, \mathsf{P}}(q)$ only depends on $c_1 \mod r H^2(\mathcal{X}_b,\mathbb{Z})$.

Theorem \ref{thm:main2} can be viewed as some kind of virtual $\mathrm{PGL}_r$--$\mathrm{SL}_r$ correspondence. For the comparison of the cohomology of the analogues of our moduli spaces on a smooth projective curve $C$, we refer to the work of Harder-Narasimhan \cite{HN}. Note however that $\mathrm{Br}(C) = 0$, so the complication is rather in dealing with torsion in $\mathrm{Pic}(C)$, which leads to questions of a different flavour. Note that we assume $H_1(X,\mathbb{Z}) = 0$ so $\mathrm{Pic}(X)$ is torsion free. It will be interesting to study analogues of Theorem \ref{thm:main2} when $H_1(X,\mathbb{Z}) \neq 0$. See \cite{Nes} for work in this direction when $X$ is the product of two curves. In the curve case, upgrade to Higgs moduli spaces is an active research area \cite{HT, GWZ, MS, HP}. In the surface case, upgrade to Higgs moduli spaces is the content of the $S$-duality conjecture. 

The intuitive idea behind the $\PGL_r$--$\SL_r$ correspondence of Theorem \ref{thm:main2} is as follows.\footnote{We thank one of the anonymous referees for this exposition.} We start with a holomorphic $\PGL_r$-bundle $Y$ on $X$ with Stiefel-Whitney class $w$. By our assumption $H^3(X,\mathbb{Z}) = 0$, it is of the form $\mathbb{P}(E)$ for a $C^\infty$ $\GL_r$-bundle $E$ on $X$ with first Chern class $\xi$ which is possibly not of Hodge type $(1,1)$. However, by the Hodge theoretic result in Proposition \ref{thm:deform-w-trivial-new} (which involves adding large $r$-multiples to $\xi$), we show that, under our assumption on the Noether-Lefschetz locus and after a small change of complex structure of $X$, one can arrange $E$ to be holomorphic and $\xi$ to be $(1,1)$. Therefore, if the original $\PGL_r$-bundle $Y$ was unobstructed, which can be arranged by de Jong's result (Theorem \ref{thm:extending-azu-algs}), then $Y = \mathbb{P}(E)$ with $E$ a holomorphic $\GL_r$-bundle in the new complex structure, and this $E$ is unique after fixing its determinant. Instead, we first completely settle the deformation invariance question in Theorem \ref{thm:main1} in the algebro-geometric category, and then use the (complex analytic) Hodge theory result in Proposition \ref{thm:deform-w-trivial-new} to find a complex structure where the $\PGL_r$ count becomes an $\SL_r$ count.

\subsection{Consequences}

\subsubsection{Application to Vafa-Witten theory} In \cite[Conj.~1.10]{GKL}, the fourth-named author and G\"ottsche-Laarakker conjecture a structure formula for the $\mathrm{SL}_r$ generating function of virtual Euler characteristics (Conjecture \ref{conj:GKL}). It appears this conjecture will be proved in a forthcoming work of Joyce as an application of his vertex algebra wall-crossing formula \cite{Joy}. Combined with Theorem \ref{thm:main2}, it leads to a structure formula for the  $\mathrm{PGL}_r$ generating function of virtual Euler characteristics as we now describe. Consider the normalized discriminant modular form
\begin{align*}
\overline{\Delta}(q) = \prod_{n=1}^{\infty} (1-q^n)^{24}.
\end{align*}
Denote by $\epsilon_r = \exp(2 \pi \sqrt{-1}/r)$ a primitive $r$th root of unity. 

Let $X$ be a smooth projective surface satisfying $H_1(X,\mathbb{Z}) = 0$ and $h^{2,0}(X)>0$. For an algebraic class $a \in H^2(X,\mathbb{Z})$, the linear system $|a|$ has a perfect obstruction theory and virtual class $|a|^{\mathrm{vir}}$ in degree $a(a-K_X) / 2$. If $|a|^{\mathrm{vir}} \neq 0$, then $a^2 = a K_X$ and $\mathrm{SW}(a) := \deg(|a|^{\mathrm{vir}})$ \cite[Prop.~6.3.1]{Moc}. This is the algebro-geometric definition of Seiberg-Witten invariants of $X$. A class $a \in H^2(X,\mathbb{Z})$ is called a Seiberg-Witten basic class when $\mathrm{SW}(a) \neq 0$.

\begin{corollary}
For any prime rank $r>1$, there exist\footnote{These universal functions only depend on $r$.} 
\[
D_0, \{D_{ij}\}_{1 \leq i \leq j \leq r-1} \in \mathbb{C}[\![q^{\frac{1}{2r}}]\!] 
\]
with the following property. Suppose $X = \mathcal{X}_0$ and $w = \widetilde{w}_0$ for a family $\mathcal{X} \to B$ satisfying the conditions of Theorem \ref{thm:main2} and $h^{2,0}(X)>0$. Fix any $\delta \in \mathbb{Z}$ such that $\delta \equiv -(r-1) w^2 - (r^2-1) \chi(\mathcal{O}_X) \mod 2r$. Then Conjecture \ref{conj:GKL} implies that the coefficient of $q^{\delta/2r}$ in $\mathsf{Z}_{(X,H),w}^{\mathrm{PGL}_r, \mathsf{Eu}}(q)$ equals the coefficient of $q^{\delta/2r}$ in
\begin{align*}
&r^{2+K_X^2 - \chi(\mathcal{O}_X)} \Bigg( \frac{1}{\overline{\Delta}(q^{\frac{1}{r}})^{\frac{1}{2}}} \Bigg)^{\chi(\mathcal{O}_X)} D_0^{K_X^2} \sum_{(a_1, \ldots, a_{r-1}) \in H^2(X,\mathbb{Z})^{r-1}}   \prod_{i} \epsilon_r^{i a_i w} \, \SW(a_i) \prod_{i \leq j} D_{ij}^{a_i a_j}.
\end{align*}
\end{corollary}
Suppose, in this corollary, $X$ is moreover minimal of general type. Then its only Seiberg-Witten basic classes are $0, K_X$ with Seiberg-Witten invariants $1,(-1)^{\chi(\mathcal{O}_X)}$ \cite[Thm.~7.4.1]{Mor}. Therefore $\mathsf{Z}_{(X,H),w}^{\mathrm{PGL}_r, \mathsf{Eu}}(q)$ only depends on $$w K_X \mod r, \quad K_X^2, \quad \chi(\mathcal{O}_X).$$ 

The novel feature of this corollary is that it provides an interpretation of the formulae in \cite{GK1,GK3,GKL} in the case $w$ has non-trivial Brauer class, i.e., it \emph{cannot} be represented by an algebraic class. By equation \eqref{eqn:PSUVWgenfun}, this essentially determines the structure of the $\mathrm{PSU}(r)$ Vafa-Witten partition function of $X$ for prime rank $r$ and $w \in H^2(X,\mu_r)$ with non-trivial Brauer class. This reduces Vafa-Witten's enigmatic $S$-duality conjecture (mathematically formulated in \cite{JiangKool2021}) to a conjecture on the $\mathrm{SU}(r)$ side.

\subsubsection{Application to $c_2^{\mathrm{min}}$ of Azumaya algebras} Let $X$ be a smooth projective surface with $H_1(X,\mathbb{Z}) = 0$ and function field $\mathbb{C}(X)$. Let $D$ be a (central) division algebra over $\mathbb{C}(X)$ of degree $r > 1$ (equivalently, an element of $\mathrm{Br}(\mathbb{C}(X))$ of order $r$). We assume $D$ lies in the image of the inclusion $\mathrm{Br}(X) \hookrightarrow \mathrm{Br}(\mathbb{C}(X))$ (basic facts on the Brauer group are reviewed in Section \ref{sec:brauer-group}, see also \cite{Salt}). 

Artin-de Jong \cite{AdJ} introduce the $\mathbb{C}$-stack $\mathfrak{A}_{c_2}$, whose groupoid over a scheme $B$ consists of Azumaya algebras $\mathcal{A}$ on $X \times B$ such that for all closed points $b \in B$, the stalk of $\mathcal{A}|_b$ over the generic point of $X$ is isomorphic to $D$ and $c_2(\mathcal{A}|_b) = c_2 \in H^4(X,\mathbb{Z}) \cong \mathbb{Z}$. This stack is algebraic and of finite type. Its coarse moduli space $\overline{\mathfrak{A}}_{c_2}$ is an algebraic space of finite type \cite[Thm.~8.7.6]{AdJ}. If $\mathfrak{A}_{c_2}$ is non-empty (i.e.~has a $\mathbb{C}$-valued point), then $c_2 \geq 0$ \cite[Thm.~7.2.1]{AdJ}. It interesting to consider the minimal value $c_2 = c_2^{\mathrm{min}}$ for which $\mathfrak{A}_{c_2}$ is non-empty, because $\mathfrak{A}_{c_2^{\mathrm{min}}}$, $\overline{\mathfrak{A}}_{c_2^{\mathrm{min}}}$ are \emph{proper} \cite[Thm.~8.7.7]
{AdJ}. Artin-de Jong proved that \cite[Cor.~7.1.5, Thm.~7.2.1]{AdJ}
\[
\mathrm{max}\big\{r^2\chi(\mathcal{O}_X) - h^0(\omega_X^{\otimes r}) - 1,0 \big\} \leq c_2^{\mathrm{min}}.
\]
We show the following.
\begin{theorem} \label{thm:main3}
Suppose $X = \mathcal{X}_0$ for a family $\mathcal{X} \to B$ satisfying the conditions of Theorem \ref{thm:main2}, and $X$ is a minimal surface of general type satisfying $h^{2,0}(X)>0$. Let $D \in \mathrm{Br}(\mathbb{C}(X))$ be a degree $r$ division algebra in the image of $\mathrm{Br}(X) \hookrightarrow \mathrm{Br}(\mathbb{C}(X))$. Then, for $r=2$, we have
\[
c_2^{\mathrm{min}} \leq 3\chi(\mathcal{O}_X) + 1.
\]
Moreover, for $r=3$ and assuming G\"ottsche's conjecture \ref{conj:Gott} for $r=3$, we have
\[
c_2^{\mathrm{min}} \leq 8\chi(\mathcal{O}_X).
\]
\end{theorem}

For minimal surfaces $X$ of general type, we have plurigenera $h^0(\omega_X^{\otimes r}) = \frac{1}{2}r(r-1)K_X^2+\chi(\mathcal{O}_X)$ \cite{Bom}. For $r=2$ and $p_g(X) = K_X^2=1$ (which implies $H_1(X,\mathbb{Z}) = 0$ \cite{Bom}), this leaves very little room: $c_2^{\mathrm{min}} \in \{4,5,6,7\}$. 

The $r=2$ case of Theorem \ref{thm:main3} is proved by combining Theorem \ref{thm:main2} with Witten's conjecture and a result of Hoffmann-Stuhler (Proposition \ref{prop:HS}). Witten's conjecture, proved by G\"ottsche-Nakajima-Yoshioka \cite{GNY3} in the algebro-geometric setup, expresses Donaldson invariants in terms of Seiberg-Witten invariants (Theorem \ref{thm:GNY}). The $r=3$ case follows similarly, but instead using the higher rank generalization of Witten's conjecture by Mari\~{n}o-Moore \cite{MM} (see also \cite{LM}), or rather its explicit form in the algebro-geometric setup due to G\"ottsche \cite{Gott}. 

This is an illustration of a general strategy described in detail in Section \ref{sec:minc2}. Suppose $X = \mathcal{X}_0$ for a family $\mathcal{X} \to B$ satisfying the conditions of Theorem \ref{thm:main2}: 
\begin{itemize}
\item Show that some virtual intersection number $\int_{[M]^{\mathrm{vir}}} \mathsf{P}$, for $M$ some moduli space of twisted sheaves on $X$ of virtual dimension $\delta$, is non-zero.
\item Then there exists an Azumaya algebra $\mathcal{A}$ on $X$ with $\mathcal{A}_{\eta} \cong D$ satisfying $c_2(\mathcal{A}) \leq \delta + (r^2-1)\chi(\mathcal{O}_X)$, where $\eta$ is the generic point of $X$.
\end{itemize}

On the other hand, non-emptiness of $\mathfrak{A}_{c_2}$ for \emph{large} $c_2$ was established by Lieblich (see Remark \ref{rem:exlargec2}). \\

\noindent \textbf{Notation and conventions.} In this paper, $X$ always denotes a smooth projective variety over $\mathbb{C}$. If $E$ is a locally free sheaf on any scheme, then $\mathbb{P}(E)$ is defined as $\Proj \Sym^\bullet E$, which is the moduli space of quotients $E \twoheadrightarrow L$ with $L$ an invertible sheaf. In this paper, we will use both singular cohomology, with respect to the complex analytic topology, and \'etale cohomology. By convention, $H^i(X, A)$ will denote singular cohomology when $A = \mathbb{Z}$, $\mathbb{Q}$, $\mathbb{R}$, $\mathbb{C}$ or $\mathbb{Z}/r\mathbb{Z}$. On the other hand, the groups $A = \mu_r$, $\mathbb{G}_m$, $\mathrm{SL}_r$, $\mathrm{GL}_r$, $\mathrm{PGL}_r$ can be viewed as sheaves of groups $\mu_r(\mathcal{O}_X)$, $\mathcal{O}_X^{\times}$, $\mathrm{SL}_r(\mathcal{O}_X)$, $\mathrm{GL}_r(\mathcal{O}_X)$, $\mathrm{PGL}_r(\mathcal{O}_X)$ in the \'etale topology and we denote the corresponding \'etale cohomology groups by $H^i(X,A)$. \\ 

\noindent \textbf{Acknowledgments.} The authors express their gratitude to Y.~Bae, P.~Brosnan, H.~Movasati, G.~Oberdieck, H.~Park, C.A.M.~Peters, F.~Reede, R.P.~Thomas, and Q.~Yin for useful discussions. M.K.~would like to provide special thanks to L.~G\"ottsche, since this paper grew out of an attempt to find a geometric motivation for \cite[Rem.~1.9]{GK3} (see also \cite[Sect.~4.3]{GKL}) in their collaboration. D.vB.~and M.K.~are supported by NWO grant VI.Vidi.192.012. M.K.~is also supported by ERC Consolidator Grant FourSurf 101087365. A.G.~and Y.J.~are supported by Simons Foundation Collaboration Grants for Mathematicians and  Y.J.~is also supported by a KU Research GO grant from the University of Kansas.


\section{Models of twisted sheaves} \label{sec:twisted-sheaves-and-models}

\subsection{Brauer-Severi varieties, Azumaya algebras, and gerbes} \label{sec:models}

In this section, we recall the notions of Brauer-Severi variety, Azumaya algebra, and gerbe, and gather various known facts about them. Let $X$ be a smooth projective variety over $\mathbb{C}$.  
\begin{definition}
    A \emph{Brauer-Severi variety of degree $r$ over $X$} is a scheme $Y \to X$ for which there exists an \'etale cover $\{U_i\}$ of $X$ on which $Y \to X$ is isomorphic to the projection $\mathbb{P}^{r-1}_{U_i} := \mathbb{P}^{r-1} \times U_i \to U_i$.
    An \emph{Azumaya algebra of degree $r$ on $X$} is a coherent sheaf of $\mathcal{O}_X$-algebras $\mathcal{A}$ for which there exists an \'etale cover $\{U_i\}$ of $X$ on which $\mathcal{A}$ is isomorphic to the matrix algebra $M_r(\mathcal{O}_{U_i})$.
\end{definition}

We will recall below that the data of a Brauer-Severi variety over $X$ is equivalent to the data of an Azumaya algebra on $X$. We first recall the following version of the Skolem-Noether theorem \cite[Prop. IV.2.3]{Milne1980}.

\begin{theorem}[Skolem-Noether] \label{thm:skolem-noether}
    For any Azumaya algebra $\mathcal{A}$ on $X$ and any automorphism $\phi \in \Aut(\mathcal{A})$, there is a covering of $X$ by Zariski open subsets $\{U_i\}$ such that $\phi|_{U_i}$ is of the form $a \mapsto u a u^{-1}$ for some $u \in \Gamma(U_i,\mathcal{A})^{\times}$ for all $i$.
\end{theorem}

\begin{corollary} \label{cor:aut-mrox}
    We have an isomorphism of sheaves ${\mathcal{A}\it{ut}}(M_r(\mathcal{O}_X)) \cong \PGL_r$ in the \'etale topology.
\end{corollary}
\begin{proof}
    The Skolem-Noether theorem implies that every automorphism of $M_r(\mathcal{O}_X)$ is locally given by conjugation by an invertible matrix, so that $\GL_r \to {\mathcal{A}\it{ut}}(M_r(\mathcal{O}_X))$ is a surjection of sheaves. The kernel of this surjection is $\mathbb{G}_m$ from which the result follows.
\end{proof}
    
The facts in the following lemma are all proved in \cite[Sect.~1.1]{Yoshioka2006}.
\begin{lemma} \label{lem:sheaf-G}
    If $\pi : Y \to X$ is a Brauer-Severi variety with relative tangent sheaf $T_{Y/X}$, then there is a unique (up to scaling) sheaf $G$ fitting in a non-trivial short exact sequence
    \begin{equation}
        0 \to \mathcal{O}_Y \to G \to T_{Y/X} \to 0.
    \end{equation}
Furthermore, the sheaf $G$ satisfies the properties $R \pi_*(G^\vee) = 0$ and the canonical map $\pi^* \pi_* (\mathcal{E}{\it{nd}}(G^\vee)) \to \mathcal{E}{\it{nd}}(G^\vee)$ is an isomorphism, where $\mathcal{E}{\it{nd}}(G^\vee) := G^\vee \otimes G$. 
\end{lemma}

We now state the following classification result on Brauer-Severi varieties and Azumaya algebras.
\begin{proposition} \label{prp:azu-bs-equiv}
    There is a canonical bijection of sets between (isomorphism classes of) Brauer-Severi varieties of degree $r$ over $X$ and Azumaya algebras of degree $r$ on $X$. Furthermore, both are equivalent to the set of $\PGL_r$-torsors, classified by $H^1(X, \PGL_r)$.
\end{proposition}

We only provide a sketch of the proof. 
\begin{proof}[Sketch of proof]
    We first describe the bijection, for the details of this part we refer to \cite{Reede2018}. If $\pi : Y \to X$ is a degree $r$ Brauer-Severi variety, then $\pi_*(\mathcal{E}{\it{nd}}(G^\vee))$ is a degree $r$ Azumaya algebra.\footnote{We observe that Reede considers the opposite algebra $\mathcal{E}{\it{nd}}(G)^{op}$, which is isomorphic to $\mathcal{E}{\it{nd}}(G^\vee)$ as sheaves of $\mathcal{O}_X$-algebras.}
    Conversely, if $\mathcal{A}$ is a degree $r$ Azumaya algebra on $X$, we can consider the moduli space of left ideals $I \subseteq \mathcal{A}$ such that $\mathcal{A}/I$ is a locally free $\mathcal{O}_X$-module of rank $r(r - 1)$, which is a degree $r$ Brauer-Severi variety over $X$.

    For a degree $r$ Brauer-Severi variety $Y \to X$, we consider the sheaf $\mathcal{I}{\it{som}}(Y, \mathbb{P}^{r-1}_X)$ of isomorphisms of schemes over $X$. This sheaf admits an action of ${\mathcal{A}\it{ut}}(\mathbb{P}^{r-1}_X) \cong \PGL_r$,\footnote{This isomorphism follows from \cite[Sect.~0.5]{MFK}.} and locally, where $Y$ is trivial, this action is free and transitive. Therefore, this is a $\PGL_r$-torsor. Similarly, for a degree $r$ Azumaya algebra $\mathcal{A}$ on $X$, one can consider the sheaf $\mathcal{I}{\it{som}}(\mathcal{A}, M_r(\mathcal{O}_X))$, which is an ${\mathcal{A}\it{ut}}(M_r(\mathcal{O}_X)) \cong \PGL_r$-torsor by Corollary \ref{cor:aut-mrox}.

    It is not hard to see that all $\PGL_r$-torsors on $X$ are obtained in this way by a gluing argument: if a cover $\{U_i\}$ trivialises a $\PGL_r$-torsor on $X$, then the transition maps on the intersections provide exactly a gluing data to glue copies of $\mathbb{P}^{r-1}$ into a Brauer-Severi variety over $X$ (and similarly for a degree $r$ Azumaya algebra on $X$). By a result on non-abelian cohomology \cite{Giraud1971}, the set of $\PGL_r$-torsors on $X$ is classified by $H^1(X, \PGL_r)$.
\end{proof}

\begin{example} \label{exa:trivial-algebras}
    Let $E$ be a locally free sheaf of rank $r$ on $X$. Then, clearly, $\pi : \mathbb{P}(E^\vee) = \mathrm{Proj}(\mathrm{Sym}^{\bullet}(E^\vee)) \to X$ is a degree $r$ Brauer-Severi variety. It is even Zariski locally isomorphic to $\mathbb{P}^{r - 1}_X$. Then $G \cong \pi^* E(1)$ and its associated Azumaya algebra is $\mathcal{E}{\it{nd}}(E^\vee)$, the endomorphism sheaf.
    Brauer-Severi varieties of this form are called \emph{trivial} because their associated category of twisted sheaves is actually untwisted, as we discuss in Section \ref{sec:twisted-sheaves}.
\end{example}

\begin{example} 
    Azumaya algebras play a prominent role in number theory. Azumaya algebras over $\mathrm{Spec} \, k$, where $k$ is a field, are precisely the central simple algebras over $k$. Thus, the quaternion algebra $\mathbb{H}$ is an Azumaya algebra over $\Spec \mathbb{R}$. There are no non-trivial examples for $k = \mathbb{C}$ --- over algebraically closed fields every Azumaya algebra is trivial. In this paper, we focus on the the case $k=\mathbb{C}(X)$, where $\mathbb{C}(X)$ denotes the function field of a smooth projective surface.
\end{example}

\begin{example}
    Just as for Azumaya algebras, there are no non-trivial Brauer-Severi varieties over $\Spec \mathbb{C}$. However, over $\Spec \mathbb{R}$ there is a non-trivial example, the variety $Y = Z(x^2 + y^2 + z^2) \subseteq \mathbb{P}^2_\mathbb{R}$. After base change to the \'etale cover $\Spec \mathbb{C} \to \Spec \mathbb{R}$, $Y$ is isomorphic to $\mathbb{P}^1_\mathbb{C}$. This Brauer-Severi variety corresponds to the quaternion algebra.
\end{example}

\begin{definition}
    An algebraic stack $\mathcal{G} \to X$ is called a \emph{gerbe on $X$} if it satisfies the following two conditions:
    \begin{enumerate}
        \item $\mathcal{G}$ admits local sections: for any\footnote{Morphisms $T \to X$ are always taken from the fppf site on $X$.} $T \to X$, there exists an \'etale cover $\{U_i \to T\}$ such that $\mathcal{G}$ admits a section over each $U_i$. 
        \item Sections of $\mathcal{G}$ are locally isomorphic: if $x$, $y$ are sections of $\mathcal{G}$ over some $T \to X$, then there is an \'etale cover $\{U_i \to T\}$ with $x|_{U_i} \cong y|_{U_i}$ for all $i$. 
    \end{enumerate}
    Let $A$ be one of the groups in ``Notation and conventions'' --- viewed as a sheaf in the \'etale topology. An \emph{$A$-gerbe} on $X$ is a gerbe $\mathcal{G}$ on $X$ together with an isomorphism $\psi_{U,x} : A(U) \cong \Aut(x)$ for each section $x$ of $\mathcal{G}$ over an \'etale open $U \to T$ for any $T \to X$. We require that $\psi$ is natural with respect to $x$ and restriction along $U$. A morphism of $A$-gerbes on $X$ is a morphism of stacks over $X$ that commutes with $\psi$.
\end{definition}

There is also a theory of gerbes when $A$ is non-abelian, which we will not need, but see \cite{Giraud1971}, where the general theory is developed.
We require the following facts from the theory of gerbes, which can all be found in \cite[Ch.~IV]{Giraud1971}.

\begin{proposition} \label{prp:properties-gerbes} 
\hfill
    \begin{enumerate}
        \item Every morphism of $A$-gerbes on $X$ is an isomorphism\footnote{We will use the term ``isomorphism'' for what is actually an \emph{equivalence} of stacks.}. Thus, the category of $A$-gerbes on $X$ is a groupoid.
        \item The stack of $A$-torsors, $[*/A] \times X$, forms an $A$-gerbe on $X$. An $A$-gerbe on $X$ is equivalent to this one if and only if it admits a section over $X$.
        \item The set of equivalence classes of $A$-gerbes on $X$ is canonically isomorphic to the sheaf cohomology group $H^2(X,A)$. The stack of $A$-torsors corresponds to the unit.
    \end{enumerate}
\end{proposition}

As a consequence of the proposition, and the fact that gerbes admit local sections, every $A$-gerbe on $X$ is \'etale locally isomorphic to the stack of $A$-torsors $[*/A] \times X$. The converse is almost true: $[*/A]$ is a group stack, and a stack over $X$ that is \'etale locally isomorphic to $[*/A] \times X$ is an $A$-gerbe on $X$ if the gluing maps are $[*/A]$-equivariant.

From now on, we will only be concerned with $\mathbb{G}_m$- and $\mu_r$-gerbes.

\begin{example}
    Consider the stack over $X$, which over $T \to X$ is determined by the groupoid of line bundles on $L$ on $T$. This is a $\mathbb{G}_m$-gerbe on $X$ and it is, in fact, the trivial $\mathbb{G}_m$-gerbe.  Indeed, it admits a global section, the trivial line bundle, and two line bundles are locally isomorphic. Moreover, there is a canonical map $\mathbb{G}_m \to \Aut(L)$ that provides the $\mathbb{G}_m$-gerbe structure.
\end{example}

\begin{example}
    A typical example of a gerbe on $X$ (which we will not need) comes from the \emph{root stack} construction: if $L$ is a line bundle on $X$, consider the moduli stack whose groupoid over $T \to X$ consists of a line bundle $M$ on $T$ together with an isomorphism $\phi : M^r \cong L|_T$. This is a $\mu_r$-gerbe on $X$.
\end{example}

\begin{example} \label{exa:gerbe-bs}
    Let $Y \to X$ be a degree $r$ Brauer-Severi variety. We consider the stack of locally free sheaves trivialising $Y$: for $T \to X$ we consider the groupoid of locally free sheaves $E$ on $T$ together with an isomorphism $\phi : Y|_T \cong \mathbb{P}(E^\vee)$ of $T$-schemes. We claim this is a $\mathbb{G}_m$-gerbe on $X$. Indeed, this stack admits \'etale local sections, since $Y$ is \'etale locally trivial. 
    Secondly, if $E$ and $F$ both trivialise $Y$, then we consider the local situation where they are both free, and now we need to produce an isomorphism. The composition $\mathbb{P}(E^\vee) \cong Y \cong \mathbb{P}(F^\vee)$ is given by an element of $\PGL_r$ \cite[Sect.~0.5]{MFK}, so it is locally an equivalence class of matrices. Any matrix in this class defines an isomorphism of $E$ and $F$ commuting with $\phi$. 
    
    If $M$ is an automorphism of $(E,\phi)$, then it has to be a scalar: this can be checked locally, where $E$ is free, and the matrix associated to $M$ acts as its representative in $\PGL_r$ on $\mathbb{P}(E^\vee)$; thus, this action is trivial if and only $M$ is a non-zero scalar. This gives us a natural isomorphism $\mathbb{G}_m \to \Aut(E,\phi)$.

    There is a variant on this example, where we also fix an isomorphism $\nu : \det E \cong \mathcal{O}_T$. This restricts the number of automorphisms of the triple $(E,\phi, \nu)$, and the result is a $\mu_r$-gerbe on $X$. Thus, associated to any degree $r$ Brauer-Severi variety over $X$ is a $\mathbb{G}_m$-gerbe and a $\mu_r$-gerbe on $X$.
\end{example}

\begin{example} \label{exa:gerbe-az}
    The same example works for degree $r$ Azumaya algebras on $X$: one can consider the stack of locally free sheaves $E$ such that $\mathcal{E}{\it{nd}}(E^\vee)$ trivialises the Azumaya algebra $\mathcal{A}$, both with and without trivialising the determinant. In this case the Skolem-Noether theorem ensures that we can use the group $\PGL_r$ in the same way as above.
    In fact, by the result of Proposition \ref{prp:azu-bs-equiv} and Example \ref{exa:trivial-algebras}, a locally free sheaf $E$ trivialises $\mathcal{A}$ if and only if it trivialises the associated Brauer-Severi variety, so the moduli stacks of these examples are isomorphic.
\end{example}

By the previous examples, every Brauer-Severi-variety $Y$ on $X$ of degree $r$ defines classes $\alpha(Y) \in H^2(X, \mathbb{G}_m)$ and $w(Y) \in H^2(X, \mu_r)$. We refer to $\alpha(Y)$ as the \emph{Brauer class} of $Y$ and to $w(Y)$ as the \emph{Stiefel-Whitney} class of $Y$. Similarly, every Azumaya algebra $\mathcal{A}$ on $X$ of degree $r$ determines classes $\alpha(\mathcal{A})  \in H^2(X, \mathbb{G}_m)$ and $w(\mathcal{A}) \in H^2(X, \mu_r)$. 
The assignments $Y \mapsto w(Y)$ and $Y \mapsto \alpha(Y)$ fit in the framework of non-abelian cohomology \cite{Giraud1971}. This theory defines groups $H^1(X, G)$ and $H^2(X, G)$ for non-abelian groups $G$ in the \'etale topology, and defines connecting homomorphisms between these groups. The definition of the connecting homomorphisms induced by the sequences
\begin{equation} \label{keyses}
    1 \to \mu_r \to \SL_r \to \PGL_r \to 1 \qquad \text{ and } \qquad 1 \to \mathbb{G}_m \to \GL_r \to \PGL_r \to 1
\end{equation}
are precisely $w$ and $\alpha$. Furthermore, these maps are compatible in the sense that the following diagram commutes:
\begin{equation} \label{diag:factor}
\xymatrix
{
H^1(X,\PGL_r) \ar^{w}[r] \ar_{\alpha}[dr] & H^2(X,\mu_r) \ar^{o}[d] \\ 
& H^2(X,\mathbb{G}_m),
}
\end{equation}
where $o$ is induced by the inclusion $\mu_r \leq \mathbb{G}_m$. This can be seen by either using that there is a morphism between the short exact sequences from \eqref{keyses}, or, alternatively, by explicitly comparing the gerbes. Note that the vertical arrow $o$ in the diagram is part of the long exact sequences induced by the Kummer sequence

\begin{equation} \label{eqn:Kummer}
\cdots \to H^1(X,\mathbb{G}_m) \to H^2(X,\mu_r) \to H^2(X,\mathbb{G}_m) \stackrel{(\cdot)^r}{\to} H^2(X,\mathbb{G}_m) \to \cdots
\end{equation}
Since $H^1(X,\mathbb{G}_m) \cong \mathrm{Pic}(X)$, we refer to the first map as $c_1 : H^1(X,\mathbb{G}_m) \to H^2(X,\mu_r)$. In the trivial case we get \cite[Lem.~1.2]{Yoshioka2006}, \cite{HuybrechtsSchroeer2003}
\begin{equation} \label{eq:w-trivial-algebra}
w(\mathbb{P}(E^\vee)) = [c_1(E) \mod r] \in H^2(X, \mu_r).
\end{equation}

By the comparison theorem \cite{Milne1980}, the \'etale cohomology group $H^2(X,\mu_r)$ equals the corresponding singular cohomology group with respect to the complex analytic topology on $X$, which we sometimes denote by $H^2(X,\mathbb{Z} / r\mathbb{Z})$. The above factorization \eqref{diag:factor} also reveals that the Brauer class of a degree $r$ Azumaya algebra on $X$ is an $r$-torsion element of $H^2(X,\mathbb{G}_m)$.

In the next section, we assume $H_1(X,\mathbb{Z})_{\mathrm{tor}} = 0$ (in the complex analytic topology). Then, by Lemma \ref{lmm:rep-w-xi}, we have 
\[
H^2(X,\mu_r) \cong H^2(X,\mathbb{Z}) / rH^2(X,\mathbb{Z}).
\]
Using the identification $H^1(X,\mathbb{G}_m) \cong \mathrm{Pic}(X)$, the map $H^1(X,\mathbb{G}_m) \to H^2(X,\mu_r)$ factors via the usual first Chern class map
\begin{displaymath} \label{diag:fac}
\xymatrix
{
\mathrm{Pic}(X) \ar[r] \ar_{c_1}[d] & H^2(X,\mathbb{Z}) / rH^2(X,\mathbb{Z}) \\
H^2(X,\mathbb{Z}) \cap H^{1,1}(X), \ar[ur] &
}
\end{displaymath}
where the diagonal arrow is the quotient map and the vertical arrow is surjective by the Lefschetz theorem on  $(1,1)$-classes. By the Kummer sequence \eqref{eqn:Kummer}, the $\mu_r$-gerbes which are trivial as $\mathbb{G}_m$-gerbes are precisely the elements of 
\[
(H^2(X,\mathbb{Z}) \cap H^{1,1}(X)) / r H^2(X,\mathbb{Z}).
\]

The cohomology class $\alpha(Y) \in H^2(X, \mathbb{G}_m)$ determines whether $Y$ is trivial (in the sense of Example \ref{exa:trivial-algebras}), because $\alpha(Y) = 0$ if and only if the gerbe from Example \ref{exa:gerbe-bs} has a global section (Proposition \ref{prp:properties-gerbes}), but this precisely happens when a trivialising locally free sheaf exists on $X$. This statement can also be seen from the previous paragraph combined with \cite[Lem.~1.4]{Yoshioka2006}.\footnote{We note, in passing, that $w(Y) = 0$ if and only if $Y \cong \mathbb{P}(E^\vee)$ for a locally free sheaf $E$ on $X$ with $\det(E) \cong \mathcal{O}_X$. However, triviality of $w(Y)$ does not play a role in this paper.}

It is not possible to recover a Brauer-Severi variety from its class in $H^2(X, \mathbb{G}_m)$; indeed, all Brauer-Severi varieties of the form $\mathbb{P}(E^\vee)$ for some locally free sheaf $E$ have trivial class. Thus, taking the associated class is not ``injective''. Instead, we have the following ``surjectivity'' results.

\begin{theorem}[de Jong] \label{thm:gabber}
    Any torsion class in $H^2(X, \mathbb{G}_m)$ is equal to $\alpha(Y)$ for some Brauer-Severi variety $Y \to X$.
\end{theorem}

\begin{theorem}[de Jong, Lieblich] \label{thm:period-index}
    Suppose $\dim(X) = 2$. Any class in $H^2(X, \mu_r)$ is equal to $w(Y)$ for some degree $r$ Brauer-Severi variety $Y \to X$.
\end{theorem}

Theorem \ref{thm:gabber} holds far more generally. We only need $X$ to admit an ample line bundle, for example, if $X$ is quasi-projective over an affine scheme (see \cite{deJong_result_of_Gabber}). On the contrary, Theorem \ref{thm:period-index} is special in the sense that it does not hold for smooth varieties of arbitrary dimension \cite{AntieauWilliams2013}. It is a consequence of the period-index theorem proved by de Jong \cite{Jong2004}. See \cite[Cor.~4.2.2.4]{Lieblich2008} on how to deduce this specific form of the theorem.

\subsection{Twisted sheaves} \label{sec:twisted-sheaves}

We review the theory of twisted sheaves from different viewpoints. Recall the sheaf $G$ from Lemma \ref{lem:sheaf-G}.

\begin{definition} \label{def:Ysheaf}
    Let $\pi : Y \to X$ be a Brauer-Severi variety. A \emph{$Y$-sheaf} is a coherent sheaf $F$ on $Y$ such that the counit $\pi^*\pi_*(F \otimes G^\vee) \to F \otimes G^\vee$ is an isomorphism. The category of $Y$-sheaves is denoted $\Coh(X,Y)$.
    Let $\mathcal{A}$ be an Azumaya algebra on $X$. An $\mathcal{A}$-module is a left module over $\mathcal{A}$ such that the underlying $\mathcal{O}_X$-module is coherent. The category of $\mathcal{A}$-modules is denoted $\Coh(X,\mathcal{A})$.
\end{definition}

\begin{proposition}[Reede] \label{prop:equivReede}
    Let $\pi : Y \to X$ be a Brauer-Severi variety over $X$ with corresponding Azumaya algebra $\mathcal{A}$ (Proposition \ref{prp:azu-bs-equiv}). Then there is a canonical equivalence between $\Coh(X,\mathcal{A})$ and $\Coh(X, Y)$.
\end{proposition}
\begin{proof}
    Let us describe the equivalence. If $F$ is a $Y$-sheaf, then $\mathcal{H}{\textit{om}}(G, F)$ is a $\mathcal{E}{\emph{nd}}(G^\vee)$-module. Hence $\pi_*(\mathcal{H}{\textit{om}}(G, F))$ is a $\pi_* \mathcal{E}{\emph{nd}}(G^\vee)$ module, but the latter is identified with $\mathcal{A}$. On the other hand, if $F$ is a $\pi_* \mathcal{E}{\emph{nd}}(G^\vee)$-module, then $\pi^*F$ is a $\pi^*\pi_* \mathcal{E}{\emph{nd}}(G^\vee) \cong \mathcal{E}{\emph{nd}}(G^\vee)$-module (Lemma \ref{lem:sheaf-G}). Hence it is a right $\mathcal{E}{\emph{nd}}(G)$-module. Then $\pi^*F \otimes_{\mathcal{E}{\emph{nd}}(G)} G$ is the corresponding $Y$-sheaf. Reede checks that this is an equivalence \cite[Lem.~1.10]{Reede2018}.
\end{proof}

The category of $Y$-sheaves only depends on $\alpha(Y) \in H^2(X, \mathbb{G}_m)$. We review the theory of coherent sheaves on gerbes and explain how to recover $\Coh(X,\mathcal{A})$, $\Coh(X,Y)$. The following result can be found in work of Lieblich \cite[Prop. 2.1.1.13, Prop. 2.2.1.6]{Lieblich_2007}. 
\begin{proposition}[Lieblich]
    \label{prp:eigensheaf-decomp}
    Let $\pi : \mathcal{G} \to X$ be a $\mathbb{G}_m$-gerbe. Every quasi-coherent sheaf $F$ on $\mathcal{G}$ admits a canonical $\mathbb{G}_m$-action. Hence, it admits a functorial weight decomposition
    \[
        F = \bigoplus_{n \in \mathbb{Z}} F_n.
    \]
    In fact, the entire category of quasi-coherent sheaves admits a weight decomposition
    \[
        \QCoh(\mathcal{G}) = \bigoplus_{n \in \mathbb{Z}} \QCoh(\mathcal{G})_n.
    \]
    Furthermore, the pullback $\pi^* : \QCoh(X) \to \QCoh(\mathcal{G})$ induces an equivalence between $\QCoh(X)$ and $\QCoh(\mathcal{G})_0$, these are the sheaves on which the action is trivial.
    For a $\mu_r$-gerbe the same result holds, except that the grading is now over $\mathbb{Z}/r\mathbb{Z}$. 
\end{proposition}

Tensor product and $\mathcal{H}{\it{om}}$ respects this decomposition, in the sense that if $F \in \QCoh(\mathcal{G})_n$ and $F' \in \QCoh(\mathcal{G})_m$, then $F \otimes F' \in \QCoh(\mathcal{G})_{n+m}$ and $\mathcal{H}{\it{om}}(F, F') \in \QCoh(\mathcal{G})_{-n+m}$. In particular, if $F$ is of pure weight, then $\mathcal{H}{\it{om}}(F, F)$ is of weight zero and thus descends to $X$.

\begin{definition}
    If $\pi : \mathcal{G} \to X$ is a $\mathbb{G}_m$-gerbe, we define the \emph{$\mathcal{G}$-twisted sheaves} as the coherent objects of $\QCoh(\mathcal{G})_1$. We denote the category of $\mathcal{G}$-twisted sheaves by $\Coh(\mathcal{G})_1$. We make analogous definitions for a $\mu_r$-gerbe on $X$, except now $1 \in \mathbb{Z}/r\mathbb{Z}$. 
\end{definition}

Let $\mathcal{A}$ be an Azumaya algebra on $X$ mapping to a $\mathbb{G}_m$-gerbe $\pi : \mathcal{G}  \to X$. Then there exists a rank $r$ locally free sheaf $E$ on $\mathcal{G}$ such that $\mathcal{E}{\it{nd}}(E^\vee) \cong \mathcal{A}|_{\mathcal{G}}$. One can verify that $E$ is in fact a twisted sheaf \cite[Cor.~2.2.2.2]{Lieblich_2007}. If $F$ is any $\mathcal{G}$-twisted sheaf, then $\mathcal{H}{\it{om}}(E, F)$ is again untwisted, i.e., it is (uniquely) in the image of $\pi^*$ by Proposition \ref{prp:eigensheaf-decomp}. Note that  $\mathcal{H}{\it{om}}(E, F)$  is a right $\mathcal{E}{\it{nd}}(E)$-module and therefore a left $\mathcal{E}{\it{nd}}(E)^{\mathrm{op}} \cong \mathcal{E}{\it{nd}}(E^\vee)$-module. The following is a consequence of Morita equivalence \cite{Lieblich_2007, Caldararu2000}:
\begin{proposition} \label{prp:gerbe-equivalence}
    Let $\mathcal{A}$ be an Azumaya algebra on $X$ mapping to a $\mathbb{G}_m$-gerbe $\mathcal{G}  \to X$. The assignment $F \mapsto \pi_* \mathcal{H}{\it{om}}(E, F)$ gives an equivalence between $\Coh(\mathcal{G})_1$ and $\Coh(X,\mathcal{A})$.\
\end{proposition}

This discussion also works for the associated $\mu_r$-gerbe, instead of the $\mathbb{G}_m$-gerbe. 

In summary: for a Brauer-Severi variety $Y \to X$ with corresponding Azumaya algebra $\mathcal{A}$ on $X$ with corresponding $\mathbb{G}_m$-gerbe $\mathcal{G}$, we have the three equivalent categories
\[
\Coh(X,Y), \quad \Coh(X,\mathcal{A}), \quad \Coh(\mathcal{G})_1,
\]
which in particular shows that (up to equivalence) the category only depends on the gerbe $\mathcal{G}$. In the next section, we view the $\mathbb{G}_m$-gerbe $\mathcal{G}$ as a Brauer class and establish in a different way that $\Coh(X,\mathcal{A})$ only depends on this Brauer class (Example \ref{exa:morita-azumaya}). In view of this discussion, it seems reasonable to call any of these equivalent  categories the \emph{category of twisted sheaves}.

\subsection{Brauer group} \label{sec:brauer-group}

The material of the previous sections is closely related to the Brauer group. We review its most important properties.

\begin{definition}
    Define two Azumaya algebras $\mathcal{A}$ and $\mathcal{A}'$ on $X$ to be \emph{Brauer equivalent} if there exist locally free sheaves $E$ and $E'$ of finite rank on $X$ such that $\mathcal{A} \otimes \mathcal{E}{\it{nd}}(E) \cong \mathcal{A}' \otimes \mathcal{E}{\it{nd}}(E')$. The resulting equivalence classes form a group, where $[\mathcal{A}] + [\mathcal{A}'] = [\mathcal{A} \otimes \mathcal{A}']$ and the inverse of $[\mathcal{A}] $ is $[\mathcal{A}^{\text{op}}]$. This group is called the \emph{Brauer group} $\Br(X)$ of $X$. For any $r>1$, we have a natural homomorphism $H^1(X,\mathrm{PGL}_r) \to \mathrm{Br}(X)$.
\end{definition}
The fact that the inverse is given by the opposite algebra follows from an alternative definition of Azumaya algebras as coherent $\mathcal{O}_X$-algebras $\mathcal{A}$, which are locally free, and for which
\begin{equation}
    \mathcal{A} \otimes \mathcal{A}^{\text{op}} \to \mathcal{E}{\it{nd}}(\mathcal{A}), \quad a \otimes a' \mapsto (x \mapsto axa')
 \end{equation}
 is an isomorphism.

\begin{proposition}
    The map $H^1(X,\mathrm{PGL}_r) \to H^2(X,\mathbb{G}_m)$ descends to an injective homomorphism $\Br(X) \to H^2(X, \mathbb{G}_m)$. Its image is exactly the collection of torsion classes in $H^2(X, \mathbb{G}_m)$.
\end{proposition}

\begin{proof}
    It suffices to show that for any Azumaya algebra $\mathcal{A}$ on $X$ and locally free sheaf $E$ of finite rank on $X$, the $\mathbb{G}_m$-gerbes associated to $\mathcal{A}$ and $\mathcal{A} \otimes \mathcal{E}{\it{nd}}(E)$ are equivalent. This is done as follows: if a locally free sheaf $E'$ on a scheme $T \to X$ trivialises $\mathcal{A}$, then $E' \otimes E$ trivialises $\mathcal{A} \otimes \mathcal{E}{\it{nd}}(E)$. This defines a $\mathbb{G}_m$-equivariant morphism between gerbes, which is therefore an isomorphism (Proposition \ref{prp:properties-gerbes}).
        
    By \eqref{keyses}, the kernel of the map $\alpha : H^1(X,\PGL_r) \to H^2(X,\mathbb{G}_m)$ is the image of $H^1(X,\mathrm{GL}_r) \to H^1(X,\PGL_r)$, which precisely consists of trivial Azumaya algebras $\mathcal{A} \cong \mathcal{E}{\it{nd}}(E)$ on $X$, from which the injectivity statement follows.

   By \eqref{diag:factor}, any element in the image of $\alpha$ is torsion. Moreover, all torsion elements of $H^2(X,\mathbb{G}_m)$ are in the image by Theorem \ref{thm:gabber}.
\end{proof}

\begin{remark}
    For a Noetherian scheme $Z$, the torsion part of $H^2(Z, \mathbb{G}_m)$ is also known as the \emph{cohomological Brauer group} of $Z$. For a regular Noetherian scheme $Z$, $H^2(Z, \mathbb{G}_m)$ is a torsion group \cite[Cor.~IV.2.6]{Milne1980}, \cite[Cor.~3.1.3.4]{Lieblich2008}, so in our setting, $\Br(X) \cong H^2(X, \mathbb{G}_m)$. 
\end{remark}

From Proposition \ref{prp:gerbe-equivalence} one can see that the category of left $\mathcal{A}$-modules only depends on the Brauer class of $\mathcal{A}$. One can also see this explicitly as follows.

\begin{example} \label{exa:morita-azumaya}
    For an Azumaya algebra $\mathcal{A}$ on $X$ and locally free sheaf $E$ of finite rank on $X$, we want to show that $\mathcal{A}$ and $\mathcal{A} \otimes \mathcal{E}nd(E)$ have equivalent categories of left modules. This follows from Morita theory \cite[Sec.~18D]{Lam1999}. Indeed, $\mathcal{A} \otimes \mathcal{E}nd(E)$ is a $(\mathcal{A}, \mathcal{A}\otimes \mathcal{E}{\it{nd}}(E))$ bimodule, which is locally free both as $\mathcal{A}$- and $\mathcal{A} \otimes \mathcal{E}{\it{nd}}(E)$-module. Explicitly, the functor 
\[
\mathrm{Coh}(X,\mathcal{A}) \to \mathrm{Coh}(X,\mathcal{A} \otimes \mathcal{E}nd(E)), \quad F \mapsto F \otimes E,
\] 
is an equivalence.
\end{example}

One might wonder if the converse is true: if $\mathcal{A}$ and $\mathcal{A}'$ are such that their categories of modules are equivalent, are they also Brauer equivalent? C\u{a}ld\u{a}raru's conjecture, now a theorem \cite{Antieau2016}, says that this is almost the case.

\begin{theorem}[Antieau] \label{thm:antieau-reconstruction}
    Let $X$ and $Y$ be quasi-compact and quasi-separated schemes over a 
    commutative ring $R$. Suppose that $\mathcal{A}_X$ and $\mathcal{A}_Y$ are 
    Azumaya algebras on $X$ respectively $Y$ such that there is an $R$-linear equivalence between 
    $\QCoh(\mathcal{A}_X)$ and $\QCoh(\mathcal{A}_Y)$. Then there is an 
    isomorphism $\varphi : X \to Y$ of $R$-schemes such that $\varphi^*\mathcal{A}_Y$ is 
    Brauer equivalent to $\mathcal{A}_X$.
\end{theorem}

\section{Moduli of twisted sheaves} \label{sec:moduli-of-twisted-sheaves}

\subsection{Chern characters}

From now on we will in addition assume that $\dim(X) = 2$, i.e., $X$ is a surface, and that $H_1(X, \mathbb{Z})$ is torsion free (in the complex analytic topology). By Poincar\'e duality and the universal coefficient theorem, this implies that all groups $H_i(X,\mathbb{Z})$, $H^i(X,\mathbb{Z})$ are torsion free.

In the next lemma, we use the assumption that $H_1(X, \mathbb{Z})$ is torsion free to construct lifts of Stiefel-Whitney classes. 
\begin{lemma} \label{lmm:rep-w-xi}
    There is an isomorphism $H^2(X, \mathbb{Z})/rH^2(X, \mathbb{Z}) \cong H^2(X, \mu_r)$ induced by the long exact sequence associated to the short exact sequence 
    \[
    0 \to \mathbb{Z} \stackrel{r \cdot}{\to} \mathbb{Z} \to \mathbb{Z}/r\mathbb{Z} \to 0,
    \]
    where $\mathbb{Z} / r\mathbb{Z} \cong \mu_r$. In other words, every $w \in H^2(X, \mu_r)$ is represented by a $\xi \in H^2(X, \mathbb{Z})$, which is unique up to translation by multiples of $r$.
\end{lemma}
\begin{proof}
    By using the long exact sequence, it suffices to show that $H^2(X, \mu_r) \to H^3(X, \mathbb{Z})$ is zero. Since the first group is torsion, it suffices to prove that $H^3(X, \mathbb{Z})$ is torsion free. This follows from Poincar\'e duality and the assumption that $H_1(X, \mathbb{Z})$ is torsion free.
\end{proof}

Another important consequence of the assumption $H_1(X,\mathbb{Z})_{\mathrm{tor}} = 0$ is that for any Brauer-Severi variety $\pi : Y \to X$, the map
$$
\pi^* : H^*(X,\mathbb{Z}) \hookrightarrow H^*(Y,\mathbb{Z})
$$
is an inclusion \cite[Lem.~1.6]{Yoshioka2006}.

The lifts of Lemma \ref{lmm:rep-w-xi} are used to define the Chern character of twisted sheaves. This definition is due to Yoshioka \cite{Yoshioka2006}. On surfaces this gives the correct answer (due to Proposition \ref{prop:integrality}, and the discussion at the end of this section), but for higher dimensional varieties another definition is required.

\begin{definition} \label{def:chern-character}
    If $\mathcal{A}$ is an Azumaya algebra on a surface $X$, choose a representing element $\xi \in H^2(X,\mathbb{Z})$ for 
    $w(\mathcal{A}) \in H^2(X,\mu_r)$ (Lemma \ref{lmm:rep-w-xi}). If $F$ is a left $\mathcal{A}$-module, we define
    \[
    \ch_{\mathcal{A}}(F) := \frac{\ch(F)}{\sqrt{\ch(\mathcal{A})}} \quad \text{and} \quad \ch^\xi_{\mathcal{A}}(F) := e^{\xi/r}\ch_{\mathcal{A}}(F).
    \]
    Here on the right we take the Chern character of coherent 
    $\mathcal{O}_X$-modules. Equivalently (via Proposition \ref{prop:equivReede}), if $\pi : Y \to X$ is the Brauer-Severi variety corresponding to $\mathcal{A} = \pi_*(\mathcal{E}{\it{nd}(G^\vee)})$ and $F$ is a $Y$-sheaf, we define
    \[
    \ch_G(F) := \frac{\ch(\pi_*(F \otimes G^\vee))}{\sqrt{\ch(\pi_*(G \otimes G^\vee))}}, \quad  \ch^\xi_{G}(F) := e^{\xi/r} \ch_G(F),
    \]
where $G$ is the sheaf defined in Lemma \ref{lem:sheaf-G}. We view these Chern characters as elements of $H^*(X,\mathbb{Q})$. The component of $\ch_G(F)$ in $H^0(X,\mathbb{Z})$ equals $\rk(F)$.
\end{definition}
 
We refer to the factor $e^{\xi/r}$ as the \emph{Huybrechts-Stellari twist} introduced in \cite{HuybrechtsStellari2005}. The class $\xi/r \in H^2(X,\mathbb{Q})$ is called a \emph{rational $B$-field}. The Huybrechts-Stellari twist has the desirable feature that $\ch^\xi_{\mathcal{A}}(F)$ has the following integrality property shown in \cite{Yoshioka2006} for K3 surfaces (but it holds for arbitrary surfaces with $H_1(X,\mathbb{Z})$ torsion free \cite{JiangKool2021}).

\begin{proposition} \label{prop:integrality}
Let $\mathcal{A}$ be an Azumaya algebra on a surface $X$. For any left $\mathcal{A}$-module $F$, choose a representing element $\xi$ for $w(\mathcal{A}) \in H^2(X,\mu_r)$ and write 
\[
\ch_{\mathcal{A}}^\xi(F) = (s, c_1, \tfrac{1}{2} c_1^2 - c_2) \in H^*(X,\mathbb{Q}).
\]
Then $c_1 \in H^2(X,\mathbb{Z})$ and $c_2 \in H^4(X,\mathbb{Z}) \cong \mathbb{Z}$.
\end{proposition} 

We wish to provide more motivation for Definition \ref{def:chern-character}. Suppose that $\mathcal{A} = \mathcal{E}{\it{nd}}(E^\vee)$ is a trivial Azumaya algebra. Then every left $\mathcal{A}$-module $M$ is of the form $F \otimes E^\vee$ for $F$ a coherent sheaf on $X$ by Example \ref{exa:morita-azumaya}. One might expect that $\ch_{\mathcal{A}}^{\xi}(M) = \ch(F)$.
And indeed, this can be verified: on a surface $X$, we have 
\begin{equation} \label{eq:surface-chern-equation}
    \ch(E) \cdot e^{-2c_1(E)/r} = \ch(E^\vee).
\end{equation}
Keeping in mind that $\xi = c_1(E)$ represents $w$, see \eqref{eq:w-trivial-algebra}, we 
can rewrite this as
\[
    \ch(F)= \frac{\ch(F \otimes E^\vee)}{\sqrt{\ch(E^\vee)^2}} = \frac{\ch(F \otimes E^\vee)}{\sqrt{\ch(\mathcal{E}{\it{nd}}(E^\vee)) e^{-2\xi/r}}} = \ch_{\mathcal{A}}^\xi(M).
\]

There is an alternative perspective. We can view our surface as a complex manifold and forget the complex structure; what remains is a topological space. There is a notion of \emph{topological Azumaya algebras} on this space. Their classes take values in the \emph{topological Brauer group}, which is the torsion part of $H^3(X, \mathbb{Z})$, which is zero in our case because we assume $H_1(X,\mathbb{Z})_{\mathrm{tor}} = 0$. Hence any topological Azumaya algebra is trivialized by a topological vector bundle, and this vector bundle is unique up to tensoring by a topological line bundle (that is, a class in $H^2(X, \mathbb{Z})$). One can use this to define Chern classes, and a similar calculation shows that they are the same as the previous definition. For details on topological Azumaya algebras, we 
refer the reader to \cite{AntieauWilliams2014a}. Similar ideas to what we have 
described also appear in \cite{Yoshioka2006, Heinloth2005}.

\subsection{Moduli spaces} \label{subsec:moduli}

Let $(X,H)$ be a smooth polarized surface and suppose $H_1(X,\mathbb{Z})_{\mathrm{tor}} = 0$. Let $\pi : Y \to X$ be a Brauer-Severi variety of degree $r$ with corresponding Azumaya algebra $\mathcal{A}$. On the one hand, one can consider moduli spaces of stable $Y$-sheaves on $X$ constructed by Yoshioka \cite{Yoshioka2006}. On the other hand, one can consider moduli spaces of generically simple torsion free left $\mathcal{A}$-modules constructed by Hoffmann-Stuhler  \cite{HoffmannStuhler2005}. We recall the main results on these moduli spaces. Since $\mathrm{Coh}(X,Y)$ and $\mathrm{Coh}(X,\mathcal{A})$ are equivalent, one expects these moduli spaces to be isomorphic as was shown by Reede \cite{Reede2018}. 

We first recall Yoshioka's result \cite{Yoshioka2006}. Let $G$ be the unique (up to scaling) non-trivial extension of $T_{Y/X}$ by $\mathcal{O}_Y$ of Lemma \ref{lem:sheaf-G}. For a $Y$-sheaf $F$, its twisted Hilbert polynomial is defined by \cite{Yoshioka2006}
$$
\chi(G,F \otimes \pi^* \mathcal{O}_X(mH)) = \chi(X,\pi_*(F \otimes G^\vee)(mH)).
$$
Suppose $F$ is torsion free. Then the above polynomial has degree 2 and we denote its leading term by $\frac{1}{2}a_2^G(F) m^2$. One defines $F$ as semistable (with respect to $H$) if
$$
\frac{\chi(\pi_*(F' \otimes G^\vee)(mH))}{a_2^G(F')}  \leq  \frac{\chi(\pi_*(F \otimes G^\vee)(mH))}{a_2^G(F)}
$$  
for all $Y$-subsheaves $0 \neq F' \subsetneq F$. 
Stability is defined analogously with $\leq$ replaced by $<$. 

\begin{definition}
For any choice of Chern character $\mathrm{ch} = (s, c_1, \frac{1}{2} c_1^2 - c_2) \in H^*(X,\mathbb{Q})$, Yoshioka proves there exists a coarse moduli space $M_Y^{H,ss}(\ch)$ parametrizing $S$-equivalence classes of semistable $Y$-sheaves $F$ with $\mathrm{ch}_G(F) = \mathrm{ch}$. Moreover, $M_Y^{H,ss}(\ch)$ is a projective scheme. Furthermore, $M_Y^{H,ss}(\ch)$ contains an open subscheme $M_Y^H(\ch)$ parametrizing isomorphism classes of stable $Y$-sheaves $F$ on $X$ with $\mathrm{ch}_G(F) = \mathrm{ch}$. 
\end{definition}

We sometimes prefer fixing Chern classes instead of Chern characters, in which case we write $M_Y^{H,ss}(s,c_1,c_2)$ and $M_Y^{H}(s,c_1,c_2)$.
Representing the Stiefel-Whitney class $w(Y)$ by $\xi \in H^2(X,\mu_r)$, we sometimes prefer to fix $\mathrm{ch}_G^\xi(F) = \mathrm{ch} = (s, c_1, \frac{1}{2} c_1^2 - c_2)$, in which case we denote the corresponding moduli spaces by
\[
M_{Y,\xi/r}^{H,ss}(\ch), \quad M_{Y,\xi/r}^{H}(\ch), \quad M_{Y,\xi/r}^{H,ss}(s,c_1,c_2), \quad M_{Y,\xi/r}^{H}(s,c_1,c_2).
\]
Then we have  $s \in \mathbb{Z}$, $c_1 \in H^2(X,\mathbb{Z})$, and $c_2 \in \mathbb{Z}$ (Proposition \ref{prop:integrality}). 

\begin{remark} \label{rem:stabauto}
Suppose we take $s$ equal to the order of $\alpha(Y) \in \mathrm{Br}(X)$. Then $s$ is the minimal rank among all (coherent) $Y$-sheaves of positive rank \cite[Lem.~3.2, Rem.~3.1]{Yoshioka2006}. This implies \emph{all} rank $s$ torsion free $Y$-sheaves are automatically stable and semistable! This is a crucial feature of the theory of twisted sheaves. 
\end{remark}

We now recall Hoffmann-Stuhler's result. They consider left $\mathcal{A}$-modules $F$ which are torsion free (as $\mathcal{O}_X$-module) and generically simple, i.e., over the generic point $\eta \in X$ the $\mathcal{A}_\eta$-module $F_\eta$ is simple. By Wedderburn's theorem, over the generic point $\mathcal{A}_\eta \cong M_n(D)$ for some division algebra $D$ over $\mathbb{C}(X)$ and some $n \in \mathbb{Z}_{>0}$, and simplicity means $F_\eta \cong D^{\oplus n}$ (with $\mathcal{A}_\eta$-module structure induced by matrix multiplication).
One crucial observation is that the algebra of $\mathcal{A}$-endomorphisms $\End_{\mathcal{A}}(F)$ is a finite-dimensional $\mathbb{C}$-algebra, which is moreover a division ring because we have an embedding \cite{HoffmannStuhler2005}
$$
\End_{\mathcal{A}}(F) \hookrightarrow \End_{\mathcal{A}_\eta}(F_\eta) \cong D^{\mathrm{op}}, 
$$
hence $\End_{\mathcal{A}}(F) = \mathbb{C}$. 
Therefore, we do not have non-trivial automorphisms. In fact, no notion of stability is required.

\begin{definition}
For any choice of Chern character $\mathrm{ch} = (s, c_1, \frac{1}{2} c_1^2 - c_2) \in H^*(X,\mathbb{Q})$, Hoffmann-Stuhler prove there exists a coarse moduli space $M_{\mathcal{A}}(\ch)$ parametrizing \emph{isomorphism classes} of generically simple torsion free left $\mathcal{A}$-modules $F$ with $\mathrm{ch}_{\mathcal{A}}(F) = \mathrm{ch}$. Moreover, $M_{\mathcal{A}}(\ch)$ is a projective scheme. 
\end{definition}
 
The property of generic simplicity implies $s = \mathrm{deg}(D) := \sqrt{\mathrm{dim}_{\mathbb{C}(X)}(D)}$ (and in particular $F$ has rank $ns^2$ as $\mathcal{O}_X$-module). Therefore, we may suppress rank $s$ from the notation. Also recall that $\mathcal{A}$ has degree $r$, so $r,n,s$ are related by $r = ns$. We sometimes prefer fixing Chern classes instead of Chern characters, in which case we write $M_{\mathcal{A}}(c_1,c_2)$.
Moreover, sometimes we fix the twisted Chern character $\mathrm{ch}^\xi_{\mathcal{A}}(F) = \mathrm{ch} = (s, c_1, \frac{1}{2} c_1^2 - c_2)$, in which case we denote the corresponding moduli space by 
\[
M_{\mathcal{A},\xi/r}(\ch), \quad M_{\mathcal{A},\xi/r}(c_1,c_2).
\] 

In the setting of Azumaya algebras, one often \emph{only} considers the case $n=1$, i.e., $\mathcal{A}_\eta \cong D$. The reason is that one can replace $\mathcal{A}$ by a Brauer equivalent Azumaya algebra with this property (\cite[Rem.~2.1]{Reede2018}) and replacing $\mathcal{A}$ by a Brauer equivalent Azumaya algebra does not change the category of left modules (Example \ref{exa:morita-azumaya}). 

Then the degree $r$ of the Azumaya algebra $\mathcal{A}$ is equal to its \emph{index} which is defined as the degree of $D$. By the period-index theorem \cite{Jong2004}, this is also the \emph{period} of $\mathcal{A}$ which is defined as the order of the Brauer class $\alpha(\mathcal{A})$.\footnote{In particular, this means $w(Y)$ is an optimal $\mu_r$-gerbe \cite[Def.~2.2.5.2]{Lieblich_2007}.}
For $n=1$, we have the following result:
\begin{theorem}[Reede]
The equivalence of categories of Proposition \ref{prop:equivReede} induces an isomorphism of moduli spaces 
$$
M_{Y}^{H}(r,c_1,c_2) \cong M_{\mathcal{A}}(c_1,c_2).
$$
In particular, all rank $r$ torsion free $Y$-sheaves are automatically stable and the space $M_{Y}^{H}(r,c_1,c_2)$ is projective and independent of $H$.
\end{theorem}
In fact, Reede identifies the moduli functors, not just the coarse spaces.

Still working under the assumption $\mathcal{A}_\eta \cong D$, Hoffmann-Stuhler consider the locus 
\[
M_{\mathcal{A}}^{\mathrm{lf}}(c_1,c_2) \subset M_{\mathcal{A}}(c_1,c_2)
\]
of locally free $\mathcal{A}$-modules of rank 1 (as $\mathcal{A}$-modules). Note that $\mathrm{Pic}(X)$ acts on the union of these moduli spaces over all $c_1, c_2$ by $F \mapsto F \otimes L$.
\begin{proposition}[Hoffmann-Stuhler] \label{prop:HS}
The map $F \to \mathcal{E}{\it{nd}}_{\mathcal{A}}(F)^{\mathrm{op}}$ gives a bijection between the closed points of 
\[
\bigsqcup_{c_1,c_2} M_{\mathcal{A}}^{\mathrm{lf}}(c_1,c_2) \Big/ \mathrm{Pic}(X)
\]
and the set of isomorphism classes of Azumaya algebras $\mathcal{B}$ on $X$ satisfying $\mathcal{B}_\eta \cong D$. Under this bijection $c_2(\mathcal{B}) =2rc_2 - (r-1) c_1^2$.
\end{proposition}
\begin{proof}
This is \cite[Prop.~4.1]{HoffmannStuhler2005}, except for the calculation of $c_2(\mathcal{B})$ with $\mathcal{B} = \mathcal{E}{\it{nd}}_{\mathcal{A}}(F)^{\mathrm{op}}$ where $F$ is a locally free $\mathcal{A}$-modules of rank 1 (as $\mathcal{A}$-module). The natural map
\[
\mathcal{A} \otimes_{\mathcal{O}_X}  \mathcal{E}{\it{nd}}_{\mathcal{A}}(F) \to \mathcal{E}{\it{nd}}(F)
\]
is an isomorphism. Since $c_1(\mathcal{A}) = 0$, we obtain
\begin{equation}
\ch(\mathcal{B}) =  \ch(\mathcal{E}{\it{nd}}_{\mathcal{A}}(F)) = \frac{\ch(\mathcal{E}{\it{nd}}(F))}{\ch(\mathcal{A})} = \frac{\ch(F)}{\sqrt{\ch(\mathcal{A})}} \Bigg( \frac{\ch(F)}{\sqrt{\ch(\mathcal{A}})} \Bigg)^\vee = r^2 + (r-1)c_1^2 -2rc_2,
\end{equation}
from which the result follows.
\end{proof}

\begin{corollary} \label{cor:HS}
If $M_{\mathcal{A}}(c_1,c_2) \neq \varnothing$, then there exists an Azumaya algebra $\mathcal{B}$ on $X$ satisfying $\mathcal{B}_\eta \cong D$ and 
\[
c_2(\mathcal{B}) \leq 2rc_2 - (r-1) c_1^2.
\]
\end{corollary}
\begin{proof}
Taking an element $F$ in $M_{\mathcal{A}}(c_1,c_2)$, its double dual (as $\mathcal{O}_X$-module) $F^{**} = \mathcal{H}{\it{om}}(\mathcal{H}{\it{om}}(F,\mathcal{O}_X),\mathcal{O}_X)$ has a natural left $\mathcal{A}$-module structure and is locally free (as $\mathcal{O}_X$-module). Consider the short exact sequence of $\mathcal{O}_X$-modules
\[
0 \to F \to F^{**} \to Q \to 0
\]
induced by the natural inclusion $F \hookrightarrow F^{**}$. Then $Q$ is 0-dimensional, so $\rk(F^{**}) = \rk(F) = r^2$ (implying $F^{**}$ is locally free of rank 1 as $\mathcal{A}$-module by Proposition \ref{prp:gerbe-equivalence}), 
$c_1(F^{**}) = c_1(F)$, and $c_2(F^{**}) = c_2(F) + c_2(Q) \leq c_2(F)$. Taking $\mathcal{B} = \mathcal{E}{\it{nd}}_{\mathcal{A}}(F^{**})^{\mathrm{op}}$, the result follows from Proposition \ref{prop:HS}. 
\end{proof}

\begin{remark}
In the previous proposition and corollary, we can also replace $M_{\mathcal{A}}^{\mathrm{lf}}(c_1,c_2)$ by $M_{\mathcal{A},\xi/r}^{\mathrm{lf}}(c_1,c_2)$, i.e., we fix Chern character twisted by the $B$-field, and obtain the same conclusions.
\end{remark}

In Section \ref{sec:minc2}, we present a method to detect the condition ``$M_{\mathcal{A}}(c_1,c_2) \neq \varnothing$'' by showing that certain intersection numbers on the moduli space are non-zero. 

\begin{remark} \label{rem:exlargec2}
Suppose $H_1(X,\mathbb{Z}) = 0$. Lieblich has shown that there exists a (sufficiently large) second Chern class $c_2 \in \mathbb{Q}$ such that $M_{\mathcal{A}}^{\mathrm{lf}}(0,c_2+k) \neq \varnothing$ for all $k \in \mathbb{Z}_{\geq 0}$ \cite[Thm.~6.2.4]{Lieblich_2009}. More precisely, let $\mathcal{G} := w(\mathcal{A}) \in H^2(X,\mu_r)$ be the optimal $\mu_r$-gerbe corresponding to $\mathcal{A}$ and let $\pi : \mathcal{G} \to X$ be the map to the coarse moduli space. 
For this remark, fix a Chern character $\ch_{\mathcal{A}} := (r,c_1,\tfrac{1}{2} c_1^2 - c_2)$ in the Chow group $A^*(X)_{\mathbb{Q}}$. Moreover, let  $\pi^* \ch_{\mathcal{A}}  =: \ch'  = (r,c_1',\tfrac{1}{2} c_1^{\prime 2} - c_2')$ in $A^*(\mathcal{G})_{\mathbb{Q}}$. Denote by $\mathbf{Tw}_{\mathcal{G}}(r,c_1',c_2')$ Lieblich's moduli stack of torsion free twisted sheaves $F$ on $\mathcal{G}$ with $\ch(F) = \ch'$. 
 Recall that $ \mathcal{A}|_{\mathcal{G}} \cong \mathcal{E}{\it{nd}}(E^\vee)$ for \emph{some} locally free twisted sheaf $E$ of rank $r$ on $\mathcal{G}$, which we can choose to satisfy $c_1(E) = 0$ (Example \ref{exa:gerbe-az}).
It is not hard to see that the equivalence of Proposition \ref{prp:gerbe-equivalence} induces a bijection between (isomorphism classes of) the $\mathbb{C}$-valued points of $\mathbf{Tw}_{\mathcal{G}}(r,c_1',c_2')$ and $M_{\mathcal{A}}(c_1,c_2)$.\footnote{Undoubtedly this bijection can be enhanced to an isomorphism between the coarse moduli spaces, but we do not need this.} The bijection also preserves locally free objects. 

Now take $c_1 = c_1'=0$. By the existence of $E$, at least one of the stacks $\mathbf{Tw}_{\mathcal{G}}(r,0,c_2')$ has a locally free object. Moreover, Lieblich shows that if $\Delta':=2rc_2' - (r-1)c_1^{\prime 2} = r^{-1}(2rc_2 - (r-1)c_1^2) = 2rc_2' = 2c_2$ is sufficiently large, then $\mathbf{Tw}_{\mathcal{G}}(r,0,c_2' + k/r)$ has a locally free object for all $k \in \mathbb{Z}_{\geq 0}$ \cite[Thm.~6.2.4]{Lieblich_2009}.\footnote{Beware of the fact that Lieblich works with a normalized degree map which is $r$ times the ordinary degree map $A_*(\mathcal{G})_{\mathbb{Q}} \to \mathbb{Q}$ \cite[Sect.~6.1.5]{Lieblich_2009}. } Hence there exists a Chern class $c_2$ such that $M_{\mathcal{A}}^{\mathrm{lf}}(0,c_2+k) \neq \varnothing$ for all $k \in \mathbb{Z}_{\geq 0}$. However, there is no a priori control over the value of $c_2$. Instead, our goal is to find values of $c_2$, with explicit lower and upper bound, for which $M_{\mathcal{A}}(0,c_2) \neq \varnothing$. 
\end{remark}

\section{Invariants of moduli of twisted sheaves}

\subsection{Obstruction theory} \label{sec:obthy}

Let $(X,H)$ be a smooth polarized surface satisfying $H_1(X,\mathbb{Z}) = 0$. Let $Y \to X$ be a Brauer-Severi variety of degree $r$ and consider the moduli spaces $M:=M_{Y}^{H}(\ch)$ of the previous section.\footnote{In this subsection, and the next one, we have ``switched off'' the $B$-field, i.e., we have not included the Huybrechts-Stellari twist for the Chern character.}
Denote by $\pi_M : Y \times M \rightarrow M$ the projection. Although a universal sheaf $\mathcal{E}$ may not exist globally on $Y \times M$, the complex 
\[
R \mathcal{H}{\it{om}}_{\pi_M}(\mathcal{E},\mathcal{E}) := R\pi_{M*} R\mathcal{H}{\it{om}}(\mathcal{E},\mathcal{E}) 
\]
exists globally on $Y \times M$ \cite{Caldararu2000, HL}. Put differently, $\mathcal{E}$ exists as a twisted sheaf on $Y \times M$. Denote the truncated cotangent complex of $M$ by $\mathbb{L}_M = \tau^{\geq -1} L_M$. The following result is well-known in the untwisted case \cite{HuybrechtsThomas2009} and shown in the twisted case in \cite{JiangKool2021}. 
\begin{proposition} \label{potM}
Fix $\ch = (s,c_1,\frac{1}{2} c_1^2 - c_2) \in H^*(X,\mathbb{Q})$. Then the moduli space $M:=M_{Y}^{H}(\ch)$ has a perfect  obstruction theory 
$$\mathbb{E} := (T_M^{\mathrm{vir}})^\vee := (R \mathcal{H}{\it{om}}_{\pi_M}(\mathcal{E},\mathcal{E})_0[1])^\vee \rightarrow \mathbb{L}_M$$
of virtual dimension
\begin{equation} \label{eqn:defvd}
\mathrm{vd}(\ch) := \mathrm{vd}(s,c_1,c_2) := 2sc_2 - (s-1)c_1^2 - (s^2-1)\chi(\mathcal{O}_X).
\end{equation}
\end{proposition}

By the work of Behrend-Fantechi \cite{BF}, we therefore obtain a virtual fundamental class of degree $\mathrm{vd}(s,c_1,c_2)$ in the Chow group of $M$
\[
[M]^{\mathrm{vir}} \in A_{\mathrm{vd}(s,c_1,c_2)}(M).
\]

We provide the following two additional perspectives on this obstruction theory.
\begin{remark}
Let $\mathcal{G} \to X$ be the $\mu_r$-gerbe corresponding to $w(Y) \in H^2(X,\mu_r)$. Then the analogue of this proposition was shown for the moduli stack $\mathbf{Tw}_{\mathcal{G}}^{s}(r,0,c_2)$ of twisted sheaves on $\mathcal{G}$ by Lieblich \cite[Prop.~6.5.1.1]{Lieblich_2009}. Another way to view this proposition is as follows. Let $\boldsymbol{M}$ be the derived stack of simple sheaves with fixed determinant on $Y$. Then its classical truncation $M^{\mathrm{cl}}$ has an obstruction theory
\[
\mathbb{E} := L_{\boldsymbol{M}}|_{M^{\mathrm{cl}}} = (R \mathcal{H}{\it{om}}_{\pi_{M^{\mathrm{cl}}}}(\mathcal{E},\mathcal{E})_0[1])^\vee \to L_{M^{\mathrm{cl}}}.
\]
Yoshioka proved that the locus of $Y$-sheaves is \emph{open} in $M^{\mathrm{cl}}$ \cite[Lem.~1.6.5]{Yoshioka2006}, so we can restrict $\mathbb{E} \to L_{M^{\mathrm{cl}}}$ to this open locus (though we are now on a stack, which is a $\mathbb{G}_m$-gerbe over the moduli scheme in Proposition \ref{potM}). 
\end{remark}

We are interested in the relative situation. Suppose $f : \mathcal{X} \to B$ is a smooth projective morphism of relative dimension 2 with connected fibres over a smooth connected variety $B$. 
We assume that one fibre (and hence all fibres) $\mathcal{X}_b$ satisfies $H_1(\mathcal{X}_b,\mathbb{Z}) = 0$.
Suppose $\pi : \mathcal{Y} \to \mathcal{X}$ is a smooth projective morphism of relative dimension $r-1$ such that the fibres over closed points $b \in B$ are Brauer-Severi varieties $\mathcal{Y}_b \to \mathcal{X}_b$. Moreover, we denote by $G_b$ the unique (up to scaling) non-trivial extension of $T_{\mathcal{Y}_b/\mathcal{X}_b}$ by $\mathcal{O}_{\mathcal{Y}_b}$ (Lemma \ref{lem:sheaf-G}).

Consider the Hodge bundles (with respect to the Zariski topology)
\[
\mathcal{H}^{2p}_{\mathrm{dR}} := \mathcal{H}_{\mathrm{dR}}^{2p}(\mathcal{X} / B) := R^{2p} f_* \Omega_{\mathcal{X} / B}^{\bullet},
\]
where $\Omega_{\mathcal{X} / B}^{\bullet}$ is the algebraic de Rham complex. Since the family $\mathcal{X} \to B$ is fixed, we suppress it from the notation. Recall that Hodge bundles behave well with respect to base change and the fibre of $\mathcal{H}^{2p}_{\mathrm{dR}}$ over a closed point $b \in B$ is 
\[
H^{2p}(\mathcal{X}_b,\Omega_{\mathcal{X}_b}^\bullet) \cong H^{2p}(\mathcal{X}_b,\mathbb{C}),
\]
where on the right hand side, we consider $\mathcal{X}_b$ with the complex analytic topology. We fix a flat section with respect to the Gau\ss-Manin connection
\[
\widetilde{v} = (\widetilde{v}_0, \widetilde{v}_1, \widetilde{v}_2) \in \bigoplus_{p=0}^{2} \Gamma(B, \mathcal{H}^{2p}_{\mathrm{dR}}).
\]
Fix a family of polarizations $\mathcal{H}$ on $\mathcal{X}$. Denote by
\[
M_{\mathcal{Y} / B} := M_{\mathcal{Y} / B}^{\mathcal{H}} (\widetilde{v})
\]
the moduli space parametrizing $\mathcal{H}_b$-stable $\mathcal{Y}_b$-sheaves $F$ with Chern character
$\ch_{G_b}(F) = \widetilde{v}_b$, for some closed point $b \in B$. This is the relative version over $B$ of the moduli spaces constructed by Yoshioka \cite{Yoshioka2006}. Note that for any closed point $b \in B$, we have a Cartesian diagram
\begin{displaymath} \label{diag:fib} 
\xymatrix 
{
M_{b} \ar@{^(->}[r] \ar[d] &  M_{\mathcal{Y}/B} \ar[d] \\
\{b\} \ar@{^(->}[r] & B
}
\end{displaymath}
where $M_b := M_{\mathcal{Y}_b}^{\mathcal{H}_b} (\widetilde{v}_b)$.

\begin{remark}
In order for these moduli spaces to be non-empty, we must have
\begin{equation} \label{eqn:filt}
\widetilde{v}_b \in \bigoplus_{p=0}^{2} H^{p,p}(\mathcal{X}_b)
\end{equation}
for \emph{some} closed point $b \in B$.
Suppose $B$ is quasi-projective. Then Deligne's invariant cycle theorem implies that \eqref{eqn:filt} holds for \emph{all} closed points $b \in B$ \cite[Prop.~11.3.5]{HodgeBook}.
\end{remark}

\begin{remark} \label{rem:notfixBfield}
Note that we do \emph{not} fix a family of $B$-fields. Such a family would be a flat section $\widetilde{\xi} \in \Gamma(B, \mathcal{H}^{2}_{\mathrm{dR}})$ such that $\widetilde{\xi}_b$ is integral for all $b \in B$. We are typically interested in families where the class $[\widetilde{\xi}_b] \in H^2(\mathcal{X}_b,\mu_r)$ has non-trivial Brauer class for some $b \in B$ and trivial Brauer class for another $b \in B$. Recall that $[\widetilde{\xi}_b] \in H^2(\mathcal{X}_b,\mu_r)$ has trivial Brauer class if and only if $\widetilde{\xi}_b$ has Hodge type $(1,1)$ modulo $r H^2(\mathcal{X}_b,\mathbb{Z})$ (Section \ref{sec:models}). If in fact $\widetilde{\xi}_b$ itself has type  $(1,1)$ for some $b \in B$ (and $B$ is quasi-projective), then it has type $(1,1)$ for all $b \in B$ by Deligne's invariant cycle theorem, which is too restrictive for our purposes. 
\end{remark}

By \cite[Thm.~4.1]{HuybrechtsThomas2009}, we have a relative obstruction theory 
\[
\phi : \mathbb{E}_{\mathcal{Y}/B} \to \mathbb{L}_{M_{\mathcal{Y}/B}/B}.
\]
For each closed point $b \in B$, consider the inclusion $j_b : \mathcal{X}_b \hookrightarrow \mathcal{X}$, then we obtain an induced map
\[
\phi_b : \mathbb{E}_b := Lj_b^* \mathbb{E}_{\mathcal{Y}/B} \to Lj_b^* \mathbb{L}_{M_{\mathcal{Y}/B}/B} \to \mathbb{L}_{M_b},
\]
which is the perfect obstruction theory of Proposition \ref{potM}. Since the base $B$ is connected, the topological number
\[
\mathrm{vd} = \mathrm{vd}(\widetilde{v}_b) \in \mathbb{Z}_{\geq 0}
\]
does not depend on the closed point $b \in B$.
We obtain a cycle class \cite[Cor.~4.3]{HuybrechtsThomas2009}
\[
[M_{\mathcal{Y}/B}]^{\mathrm{vir}} \in A_{\mathrm{vd}+ \dim(B)}(M_{\mathcal{Y}/B})
\]
such that for all closed points $b \in B$ we have
\begin{equation} \label{eqn:Gysinpull}
i_b^! [M_{\mathcal{Y}/B}]^{\mathrm{vir}} = [M_b]^{\mathrm{vir}} \in A_{\mathrm{vd}}(M_b),
\end{equation}
where $i_b^! : A_*(M_{\mathcal{Y}/B}) \to A_*(M_{b})$ is the refined Gysin pull-back \cite{Fulton} for the Cartesian diagram \eqref{diag:fib}.

\subsection{$\mathrm{SL}_r$ and $\mathrm{PGL}_r$ generating function} \label{sec:genfun}

Let $R$ be a commutative $\mathbb{Q}$-algebra e.g.~$\mathbb{Q}$ or a polynomial algebra over $\mathbb{Q}$. Fix $v \in \mathbb{Z}_{\geq 0}$, and $\boldsymbol{a} = (a_1, \ldots, a_N) \in \{1,2\}^N$ for some $N \in \mathbb{Z}_{\geq 0}$. Let $(\underline{\alpha_0},\underline{\alpha_1},\ldots,\underline{\alpha_N})$ be a list of variables and let $\underline{r}$ be a further variable. Let  $\mathsf{P}_{v,\boldsymbol{a}}$ be a formal power series, with coefficients in $R(\underline{r})$, in the following \emph{formal symbols}
\begin{align*}
&\pi_{M*}\Big( \pi_X^* \underline{\alpha_i} \cap \ch_{k}( \underline{\mathcal{E}} \otimes \det(\underline{\mathcal{E}})^{-\frac{1}{\underline{r}}})\Big), \quad \pi_{M*}\Big( \pi_X^* \underline{\alpha_i c_1(X)} \cap \ch_{k}( \underline{\mathcal{E}} \otimes \det(\underline{\mathcal{E}})^{-\frac{1}{\underline{r}}})\Big), \\ 
&\pi_{M*}\Big( \pi_X^* \underline{\alpha_i c_2(X)} \cap \ch_{k}( \underline{\mathcal{E}} \otimes \det(\underline{\mathcal{E}})^{- \frac{1}{\underline{r}}})\Big), \quad \pi_{M*}\Big( \pi_X^* \underline{\alpha_i c_1(X)^2} \cap \ch_{k}( \underline{\mathcal{E}} \otimes \det(\underline{\mathcal{E}})^{-\frac{1}{\underline{r}}})\Big), \\ 
&c_j(R\mathcal{H}{\it{om}}_{\pi_M}(\underline{\mathcal{E}}, \underline{\mathcal{E}})[1]),
\end{align*}
where $i=0,\ldots, N$ and $k,j \in \mathbb{Z}_{\geq 0}$, and $\underline{\alpha_i}, \underline{\alpha_i c_1(X)}, \ldots, \underline{\mathcal{E}}$ are viewed as formal variables. Define degrees $\deg \, c_j(\cdot) = j$ and
\begin{align*}
\deg \,\pi_{M*}\Big( \pi_X^* \underline{\alpha_i} \cap \ch_{k}( \underline{\mathcal{E}} \otimes \det(\underline{\mathcal{E}})^{- \frac{1}{\underline{r}}})\Big) &= a_i + k-2, \\
\deg \,\pi_{M*}\Big( \pi_X^* \underline{\alpha_i c_1(X)} \cap \ch_{k}( \underline{\mathcal{E}} \otimes \det(\underline{\mathcal{E}})^{- \frac{1}{\underline{r}}})\Big) &= a_i + k - 1,  \\
\deg \,\pi_{M*}\Big( \pi_X^* \underline{\alpha_i c_2(X)} \cap \ch_{k}( \underline{\mathcal{E}} \otimes \det(\underline{\mathcal{E}})^{- \frac{1}{\underline{r}}})\Big) &= a_i + k, \\
\deg \,\pi_{M*}\Big( \pi_X^* \underline{\alpha_i c_1(X)^2} \cap \ch_{k}( \underline{\mathcal{E}} \otimes \det(\underline{\mathcal{E}})^{-\frac{1}{\underline{r}}})\Big) &= a_i + k, 
\end{align*}
where $i=0,\ldots, N$ and we take $a_0:=0$. We assume that the formal power series  $\mathsf{P}_{v,\boldsymbol{a}}$ has only finitely many terms in each degree. We refer to $\mathsf{P}_{v,\boldsymbol{a}}$ as a \emph{formal insertion}. The reader who is only interested in virtual Euler characteristics can take $\mathsf{P}_v = c_v(R\mathcal{H}{\it{om}}_{\pi_M}(\underline{\mathcal{E}}, \underline{\mathcal{E}})[1])$.

Keeping $\boldsymbol{a}$, $\underline{r}$ fixed as above, we now consider a \emph{sequence of formal insertions} $\mathsf{P} := \{\mathsf{P}_{v,\boldsymbol{a}}\}_{v \geq 0}$.\footnote{In \cite[App.~A]{GKL}, we define a further multiplicative property on sequences of formal insertions, but multiplicativity does not play a role in this paper.} Let $r >1$, let $(X,H)$ be a smooth polarized surface satisfying $H_1(X,\mathbb{Z}) = 0$, and $\boldsymbol{\alpha} = (\alpha_1, \ldots, \alpha_N) \in H^{*}(X,\mathbb{Q})^N$ algebraic classes with $\deg(\alpha_i) = a_i$ for all $i=1, \ldots, N$. \\

\noindent \textbf{$\mathrm{SL}_r$-invariants.} As in the introduction, denote by $M:=M_X^H(r,c_1,c_2)$ the Gieseker-Maruyama-Simpson moduli space of rank $r$ Gieseker $H$-stable sheaves on $X$ with Chern classes $c_1,c_2$ \cite{HL}. Suppose, for the moment, that there exists a universal sheaf $\mathcal{E}$ on $X \times M$ such that the line bundle $\det(\mathcal{E})$ has an $r$th root. We drop these two assumptions in Remark \ref{rem:dropass}. Take
$$
v := \rk(R\mathcal{H}{\it{om}}_{\pi_M}(\mathcal{E},\mathcal{E})[1]) + \chi(\mathcal{O}_X).
$$
Then the evaluation $\mathsf{P}_{v,\boldsymbol{a}}(r,X,\boldsymbol{\alpha};M,\mathcal{E})$ is defined as the cohomology class on $M$ obtained from $\mathsf{P}_{v,\boldsymbol{a}}$ by substituting\footnote{There is some redundancy here. Obviously $\alpha_i c_2(X) = \alpha_i c_1(X)^2 = 0$ for all $i=1, \ldots, N$ for degree reasons. We nonetheless allow these classes for notational convenience.} 
\begin{align*}
&\underline{\alpha_0} = 1, \quad \underline{\alpha_0 c_1(X)} = c_1(X), \quad \underline{\alpha_0 c_2(X)} = c_2(X), \quad \underline{\alpha_0 c_1(X)^2} = c_1(X)^2,  \quad \underline{\mathcal{E}} = \mathcal{E}, \\ 
&\underline{r} = r, \quad \underline{\alpha_i} = \alpha_i, \quad \underline{\alpha_i c_1(X)} = \alpha_i c_1(X), \quad \underline{\alpha_i c_2(X)} = \alpha_i c_2(X), \quad \underline{\alpha_i c_1(X)^2} = \alpha_i c_1(X)^2
\end{align*}
for all $i=1, \ldots, N$, where $\pi_X : X \times M \to M$, $\pi_M : X \times M \to M$ denote the projections.

\begin{definition} \label{def:GLgenfun}
Fix a first Chern class $c_1 \in H^2(X,\mathbb{Z})$. 
Suppose there are no rank $r$ strictly Gieseker $H$-semistable sheaves on $X$ with first Chern class $c_1$ (this is the case, for example, when $\gcd(r,c_1H)=1$).
We define the $\mathrm{SL}_r$ generating function of $(X,H), c_1$ associated to formal insertions $\mathsf{P} := \{\mathsf{P}_{v,\boldsymbol{a}}\}_{v \geq 0}$ by
\[
\mathsf{Z}_{(X,H),c_1}^{\mathrm{SL}_r, \mathsf{P}}(q) := \sum_{c_2 \in \mathbb{Z}} q^{\frac{\mathrm{vd}(r,c_1,c_2)}{2r}} \int_{[M_X^H(r,c_1,c_2)]^{\mathrm{vir}}} \mathsf{P}_{\mathrm{vd}(r,c_1,c_2),\boldsymbol{a}}(r,X,\boldsymbol{\alpha};M_X^H(r,c_1,c_2),\mathcal{E}),
\] 
where $\mathrm{vd}(r,c_1,c_2)$ was defined in \eqref{eqn:defvd}.
\end{definition}

\noindent \textbf{$\mathrm{PGL}_r$-invariants.} Fix a class $w \in H^2(X,\mu_r)$. By Theorem \ref{thm:period-index}, there exists a degree $r$ Brauer-Severi variety $Y \to X$ such that $w(Y) = w$. Choose a representative $\xi \in H^2(X,\mathbb{Z})$ of $w$, i.e., $w = [\xi]$ (Lemma \ref{lmm:rep-w-xi}). As in Section \ref{subsec:moduli}, denote by $M:=M_{Y,\xi/r}^H(r,\xi,c_2)$ Yoshioka's moduli space of $H$-stable $Y$-sheaves $F$ satisfying
\[
\ch_G^{\xi}(F) = (r,\xi,\tfrac{1}{2} \xi^2 - c_2).
\]
Note that we have ``switched on'' the $B$-field so that $c_2 \in \mathbb{Z}$ (Proposition \ref{prop:integrality}). Also note that we choose the rank of $F$ equal to the degree of $Y$ and ``twisted'' first Chern class equal to $\xi$.\footnote{In particular, $M_{Y,\xi / r}^H(r,\xi,c_2) = M_Y^H(r,0,c_2+\frac{1-r}{2r} \xi^2)$.} Suppose, for the moment, that there exists a universal sheaf $\mathcal{E}$ on $Y \times M$ such that the line bundle $\det(\mathcal{E})$ has an $r$th root. We drop these two assumptions in Remark \ref{rem:dropass}. Take
$$
v := \rk(R\mathcal{H}{\it{om}}_{\pi_M}(\mathcal{E},\mathcal{E})[1]) + \chi(\mathcal{O}_X).
$$
The evaluation $\mathsf{P}_{v,\boldsymbol{a}}(r,X,\boldsymbol{\alpha};M,\mathcal{E})$ is defined as the cohomology class on $M$ obtained from $\mathsf{P}_{v,\boldsymbol{a}}$ by substituting
\[
\underline{\mathcal{E}} \otimes \det(\underline{\mathcal{E}})^{-\frac{1}{\underline{r}}} = \pi_*(\mathcal{E} \otimes \det(\mathcal{E})^{-\frac{1}{r}})
\]
where $\pi$ is the base change of $\pi : Y \to X$ to $Y \times M$, and furthermore by substituting
\begin{align*}
&\underline{\alpha_0} = 1, \quad \underline{\alpha_0 c_1(X)} = c_1(X), \quad \underline{\alpha_0 c_2(X)} = c_2(X), \quad \underline{\alpha_0 c_1(X)^2} = c_1(X)^2, \\ 
&\underline{r} = r, \quad \underline{\alpha_i} = \alpha_i, \quad \underline{\alpha_i c_1(X)} = \alpha_i c_1(X), \quad \underline{\alpha_i c_2(X)} = \alpha_i c_2(X), \quad \underline{\alpha_i c_1(X)^2} = \alpha_i c_1(X)^2
\end{align*}
for all $i=1, \ldots, N$, where $\pi_X : X \times M \to M$, $\pi_M : X \times M \to M$ denote the projections.

\begin{definition} \label{def:PGLgenfun}
Fix $w \in H^2(X,\mu_r)$. Choose a degree $r$ Brauer-Severi variety $Y \to X$ with $w(Y) = w$ and choose a representative $\xi \in H^2(X,\mathbb{Z})$ of $w$. 
Suppose there are no rank $r$ strictly $H$-semistable $Y$-sheaves $F$ with $\mathrm{ch}_{G}^{\xi}(F) = (r,\xi,\frac{1}{2} \xi^2 - c_2)$ for any $c_2$ (this is the case, for example, when the Brauer class $o(w) \in H^2(X,\mathbb{G}_m)$ has order $r$ by Remark \ref{rem:stabauto}). 
We define the $\mathrm{PGL}_r$ generating function of $(X,H), w$ associated to formal insertions $\mathsf{P} := \{\mathsf{P}_{v,\boldsymbol{a}}\}_{v \geq 0}$ by
\[
\mathsf{Z}_{(X,H),w}^{\mathrm{PGL}_r, \mathsf{P}}(q) := \sum_{c_2 \in \mathbb{Z}} q^{\frac{\mathrm{vd}(r,\xi,c_2)}{2r}} \int_{[M_{Y,\xi/r}^H(r,\xi,c_2)]^{\mathrm{vir}}} \mathsf{P}_{\mathrm{vd}(r,\xi,c_2),\boldsymbol{a}}(r,X,\boldsymbol{\alpha};M_{Y,\xi/r}^H(r,\xi,c_2),\mathcal{E}),
\]
where $\mathrm{vd}(r,\xi,c_2)$ was defined in \eqref{eqn:defvd}. As we will see below (Proposition \ref{prop:indep}), this generating function does \emph{not} depend on the choice of $Y, \xi$ on the right hand side --- justifying the notation. 
\end{definition}

\begin{remark} \label{rem:dropass}
In each of the above two settings, in general a universal sheaf $\mathcal{E}$ on $X \times M$ (respectively $Y \times M$) may only exist \'etale locally. However, note that the following complex and sheaf always exist globally by \cite[Thm.~2.2.4]{Caldararu2000} (see also \cite[Sect.~10.2]{HL})
\[
R\mathcal{H}{\it{om}}_{\pi_M}(\mathcal{E},\mathcal{E}), \quad \mathcal{E}^{\otimes r} \otimes \det(\mathcal{E})^{-1}. 
\]
Note furthermore that for any class $V \in 1+K^0(X \times M)$, we can define an $r$th root operation
\[
\sqrt[r]{V} \in K^0(X \times M)_{\mathbb{Q}}
\]
precisely as in \cite[Lem.~5.1]{OT} (which handles the $r=2$ case, but the arguments generalize). Then we set
\[
\ch(\mathcal{E} \otimes (\det(\mathcal{E}))^{-\frac{1}{r}}) := \mathrm{ch}\Big( r\sqrt[r]{\frac{1}{r^r} \cdot \mathcal{E}^{\otimes r} \otimes \det(\mathcal{E})^{-1} } \Big),
\]
and similarly for $\ch(\pi_*(\mathcal{E} \otimes (\det(\mathcal{E}))^{-\frac{1}{r}}))$ for a degree $r$ Brauer-Severi variety $\pi : Y \to X$.
\end{remark}

We will now prove that the generating function $\mathsf{Z}^{\mathrm{PGL}_r, \mathsf{P}}_{(X,H), w}(q)$ does not depend on the choice of Brauer-Severi variety $Y$ with Stiefel-Whitney class $w$, or representative $\xi$ of $w$. Here we will make crucial use of \cite[Lem.~3.5]{Yoshioka2006}, which establishes the relevant isomorphism of moduli spaces. We extend this result to also include an isomorphism of virtual tangent bundles. Furthermore, we prove \cite[Lem.~3.5]{Yoshioka2006} in a slightly different way, and we will provide more details and include some intermediate steps that are useful on their own. We first recall Yoshioka's setup.

Suppose $p_1 : Y_1 \to X$ and $p_2 : Y_2 \to X$ are two Brauer-Severi varieties with the same Brauer class $o(w(Y_1)) = o(w(Y_2))$ (we will later assume $Y_1$, $Y_2$ have the same degree and $w(Y_1) = w(Y_2)$). Denote the projections $Y_1 \times_X Y_2 \to Y_1$ and $Y_1 \times_X Y_2 \to Y_2$ by $\pi_1$ and $\pi_2$, respectively. Yoshioka shows that there exists a line bundle $L$ on $Y_1 \times_X Y_2$ such that the Fourier-Mukai transform 
\[
\Xi : \Coh(X, Y_1) \to \Coh(X, Y_2), \quad \Xi(F) = \pi_{2*}(\pi_1^*F \otimes L)
\]
defines an equivalence of categories \cite[Lem.~1.7]{Yoshioka2006}. Note that we can always replace $L$ by $L \otimes \pi_2^*p_2^*P$ with $P$ a line bundle on $X$, something we will do later. We first record the following identities.

\begin{lemma} \label{lem:y-sheaf-identity}
    Let $F$ be an $Y_1$-sheaf. The canonical map 
    $$
    \pi_2^*\Xi(F) = \pi_2^*\pi_{2*}(\pi_1^*F \otimes L) \to \pi_1^*F \otimes L
    $$
    is an isomorphism, i.e., $\pi_1^*F \otimes L$ is globally generated relative to $\pi_2$. If $F'$ is another $Y_1$-sheaf that is locally free, we have that 
    $$
    \pi_2^*(\Xi(F')^\vee \otimes \Xi(F)) \cong \pi_1^*((F')^\vee \otimes F).
    $$
\end{lemma}
\begin{proof}
    We can check the first statement \'etale locally, where the Brauer-Severi varieties are trivial. Then, the sheaves $\mathcal{O}_{Y_1}(1)$ and $\mathcal{O}_{Y_2}(1)$ exist. Yoshioka has already shown that $F(-1)$ is globally generated relative to $p_1$. Then, by base change, $\pi_1^*(F(-1))$ is globally generated relative to $\pi_2$, and also, $\pi_1^*(F(-1)) \otimes \pi_2^*Q$ is globally generated relative to $\pi_2$ for any line bundle $Q$. The claim now follows from the construction of $L$ by Yoshioka: on this \'etale cover it is given by $\mathcal{O}_{Y_1}(-1) \boxtimes (\mathcal{O}_{Y_2}(1) \otimes p_2^*P)$ for $P$ any line bundle on $X$.

    The second statement directly follows from the first.
\end{proof}

\begin{proposition} \label{prop:indep}
 The generating function $\mathsf{Z}_{(X,H),w}^{\mathrm{PGL}_r, \mathsf{P}}(q)$ does not depend on the choice of degree $r$ Brauer-Severi variety $Y \to X$ with $w(Y) = w$ or representative $\xi \in H^2(X,\mathbb{Z})$ of $w$.
\end{proposition}

\begin{proof}
    We start with the independence of $Y$. Using the above notation, let $Y_1, Y_2$ be degree $r$ Brauer-Severi varieties on $X$ with $w(Y_1) = w(Y_2) = w$. We will first show that $\Xi$ induces an isomorphism between moduli spaces, and then that it preserves the obstruction theory and the generating function.
    
    We first want to prove that $\Xi$ preserves the Chern character, in order to have a chance at showing it preserves the moduli spaces. 
    Denote by $G_i$ the unique (up to scaling) non-trivial extension of $T_{Y_i/X}$ by $\mathcal{O}_{Y_i}$ (Lemma \ref{lem:sheaf-G}).
    We consider the equality $\ch_{G_1}(F) = \ch_{G_2}(\Xi(F))$ for any $Y_1$-sheaf $F$ of positive rank. It suffices to show this equality on $Y_1 \times_X Y_2$ modulo $H^{\geq 6}$, since $Y_1 \times_X Y_2 \to X$ induces an injective map on cohomology. Using the fact that $F, G_1$ are $Y_1$-sheaves and $\Xi(F), G_2$ are $Y_2$-sheaves (Definition \ref{def:Ysheaf}), it suffices to prove the following equality modulo $H^{\geq 6}$
    \[
        \ch \left((\pi_1^* F \otimes \pi_1^*G_1^\vee)^{\otimes 2} \otimes \pi_2^*G_2 \otimes \pi_2^*G_2^\vee\right) = \ch \left( (\pi_2^*(\Xi F) \otimes \pi_2^*G_2^\vee)^{\otimes 2} \otimes \pi_1^*G_1 \otimes \pi_1^*G_1^\vee \right).
    \]
    By Lemma~\ref{lem:y-sheaf-identity}, this is equivalent to the following equation modulo $H^{\geq 6}$
    \begin{align*}
        \ch &\left( (\pi_1^* F)^{\otimes 2} \otimes \pi_1^*G_1^\vee \otimes \pi_2^*(\Xi G_1)^\vee \otimes \pi_2^*G_2 \otimes \pi_2^*G_2^\vee \right)\\
        &= \ch \left( (\pi_1^* F)^{\otimes 2} \otimes (\pi_2^*G_2^\vee)^{\otimes 2} \otimes \pi_2^*(\Xi G_1) \otimes \pi_1^*G_1^\vee \right),
    \end{align*}
    which can then be rearranged to give
    \[
        \ch( (\pi_1^*F)^{\otimes 2} \otimes \pi_1^*G_1^\vee \otimes \pi_2^*G_2^\vee) \cdot \pi_2^*\ch( \Xi(G_1)^\vee \otimes G_2 - G_2^\vee \otimes \Xi(G_1)) = 0.
    \]
    Thus, to obtain the Chern character equality, it suffices to show that $\ch(\Xi(G_1)^\vee \otimes G_2) = \ch(G_2^\vee \otimes \Xi(G_1))$ modulo $H^{\geq 6}(Y_2,\mathbb{Q})$. Since, $\mathrm{ch}_0$ and $\mathrm{ch}_2$ already agree, we only have to worry about the $c_1$. We show that we can replace $L$ by $L \otimes \pi_2^*p_2^*P$ for some $P \in \Pic(X)$ such that $c_1(G_2) = c_1(\Xi(G_1))$, which proves what we want. Changing $L$ by $L \otimes \pi_2^* p_2^* P$ replaces $\Xi$ by $F \mapsto p_2^*P \otimes \Xi(F)$, so, in particular, it replaces $c_1(\Xi(G_1))$ by $c_1(\Xi(G_1)) + rp_2^*c_1(P)$. Thus, in order to find the desired $P$, we have to show that
    \begin{equation}\label{eq:chern-in-coh}
    c_1(G_2) - c_1(\Xi(G_1)) \in r \cdot p_2^*(H^2(X, \mathbb{Z}) \cap H^{1,1}(X)).
    \end{equation}
    The equality $p_2^*p_{2*}(\Xi(G_1) \otimes G_2^\vee) \cong \Xi(G_1) \otimes G_2^\vee$, which holds since $\Xi(G_1)$ is a $Y_2$-sheaf, already implies that $r(c_1(G_2)-c_1(\Xi(G_1)))$ lies in $r \cdot p_2^*(H^2(X, \mathbb{Z}) \cap H^{1,1}(X))$, which implies that
    \[
        c_1(G_2) - c_1(\Xi(G_1)) \in p_2^*(H^2(X, \mathbb{Z}) \cap H^{1,1}(X)).
    \]
    To get the factor $r$, we see that the image of $c_1(G_2) - c_1(\Xi(G_1))$ in $H^2(Y_2, \mu_r)$ is $p_2^*w - p_2^*w = 0$, by the assumption $w(Y_1) = w(Y_2) = w$ and \cite[Lem.~1.3, 1.8]{Yoshioka2006}. Using the explicit form of $H^2(Y_2, \mathbb{Z}) = p_2^*H^0(X,\mathbb{Z}) \oplus p_2^*H^2(X, \mathbb{Z})$, we see that \eqref{eq:chern-in-coh} is satisfied. This shows that we can pick $\Xi$ such that it preserves the Chern character. For the rest of the proof, we will use this $\Xi$.

    It is now not hard to show that 
    $$
    \frac{\chi(X,p_{1*}(F \otimes G_1)(mH))}{\rk(F)r} = \frac{\chi(X,p_{2*}(\Xi(F) \otimes G_2)(mH))}{\rk(F) r} + c,
    $$
    where $c$ is a constant that does not depend on $F$. Therefore $\Xi$ preserves stability.
    Since $\Xi$ preserves the Chern character and the stability condition, we obtain an isomorphism
    \[
    \phi : M_{Y_1,\xi/r}^H(r,\xi,c_2) \cong M_{Y_2,\xi/r}^H(r,\xi,c_2)
    \]
    for any lift $\xi$ of $w$ and any $c_2 \in \mathbb{Z}$. 

    Suppose there exists a universal sheaf $\mathcal{E}_1$ on $Y_1 \times M_1$, where $M_i:=M_{Y_i,\xi/r}^H(r,\xi,c_2)$. Then
    \[
    \mathcal{E}_2 := \phi_* \Xi(\mathcal{E}_1)
    \]
    is a universal sheaf on $Y_2 \times M_2$, where we denote the base-changed versions of $\phi, \Xi$ by the same symbol. By the second identity of Lemma~\ref{lem:y-sheaf-identity}, we have
    \[
    Rp_{2*} R{\mathcal{H}\it{om}}(\Xi(\mathcal{E}_1),\Xi(\mathcal{E}_1)) \cong Rp_{1*} R{\mathcal{H}\it{om}}(\mathcal{E}_1, \mathcal{E}_1). 
    \]
    This readily implies $\phi^* T_{M_2}^{\mathrm{vir}} \cong T_{M_1}^{\mathrm{vir}}$. Note that this reasoning does not require the universal sheaf $\mathcal{E}_1$ to exist globally on $Y_1 \times M_1$.
    This equality, and Siebert's formula \cite[Thm.~4.6]{Sie}, imply that the virtual classes defined by both moduli problems agree\footnote{We note that the comparison of virtual tangent bundles and fundamental classes \eqref{eqn:virclequal} was used but not explicitly carried out in \cite{JiangKool2021}.}
    \begin{equation} \label{eqn:virclequal}
    \phi^* [M_{Y_2,\xi/r}^H(r,\xi,c_2)]^{\mathrm{vir}} = [M_{Y_1,\xi/r}^H(r,\xi,c_2)]^{\mathrm{vir}}.
    \end{equation}
    Note that it is not necessary to establish that 
    \begin{displaymath}
    \xymatrix
    {
    (T_{M_1}^{\mathrm{vir}})^\vee \ar[d] \ar^{\cong}[r] & (\phi^* T_{M_2}^{\mathrm{vir}})^\vee \ar[d] \\
    \mathbb{L}_{M_1} \ar^{\cong}[r] & \phi^* \mathbb{L}_{M_2}.
    }
    \end{displaymath}
    commutes, and we will therefore not do this.

    To complete this part of the proof, we should check that the invariants themselves agree. Lemma~\ref{lem:y-sheaf-identity} implies that
    \begin{align*}
        \pi_2^*\phi^*(\mathcal{E}_2^{\otimes r} \otimes \det(\mathcal{E}_2)^{-1}) & \cong (\pi_2^* \phi^*\mathcal{E}_2)^{\otimes r} \otimes \det(\pi_2^* \phi^*\mathcal{E}_2)^{-1})\\
        & \cong (\pi_2^*\Xi(\mathcal{E}_1))^{\otimes r} \otimes \det(\pi_2^*\Xi(\mathcal{E}_1))^{-1}\\
        & \cong (\pi_1^*\mathcal{E}_1 \otimes L)^{\otimes r} \otimes \det(\pi_1^*\mathcal{E}_1 \otimes L)^{-1}\\
        & \cong \pi_1^*( \mathcal{E}_1^{\otimes r} \otimes \det(\mathcal{E}_1)^{-1}).
    \end{align*}
    Note that it is crucial that we work with the ``normalized'' universal sheaves. This gives us that, for any $\alpha \in H^*(X, \mathbb{Q})$ and $k \in \mathbb{Z}$ we have
    \[
    \pi_{M_1*} \Big( \pi_X^* \alpha \cap \mathrm{ch}_k(p_{1*} (\mathcal{E}_1 \otimes \det(\mathcal{E}_1)^{-\frac{1}{r}} )) \Big) = 
    \phi^* \pi_{M_2*} \Big( \pi_X^* \alpha \cap \mathrm{ch}_k(p_{2*} (\mathcal{E}_2 \otimes \det(\mathcal{E}_2)^{-\frac{1}{r}} )) \Big),
    \]
    where $\pi_X : M_i \times X \to X$ is the projection. We conclude that, for any $\boldsymbol{a}$, $\boldsymbol{\alpha}$ as in Definition \ref{def:PGLgenfun}, we have
    \begin{align*}
    \int_{[M_{Y_1,\xi/r}^H(r,\xi,c_2)]^{\mathrm{vir}}} \mathsf{P}_{\mathrm{vd}(r,\xi,c_2),\boldsymbol{a}}(r,X,\boldsymbol{\alpha};M_{Y_1,\xi/r}^H(r,\xi,c_2),\mathcal{E}_1) = \\
    \int_{[M_{Y_2,\xi/r}^H(r,\xi,c_2)]^{\mathrm{vir}}} \mathsf{P}_{\mathrm{vd}(r,\xi,c_2),\boldsymbol{a}}(r,X,\boldsymbol{\alpha};M_{Y_2,\xi/r}^H(r,\xi,c_2),\mathcal{E}_2).
    \end{align*}
    
    The second part of the proposition follows \cite[Prop.~3.1]{JiangKool2021}. Let $\pi : Y \rightarrow X$ be a degree $r$ Brauer-Severi variety with $w(Y)=w$ and let $\xi, \xi' \in H^2(X,\mathbb{Z})$ be two representatives of $w$ (Lemma \ref{lmm:rep-w-xi}). Then $\xi' = \xi+r\gamma$ for some $\gamma \in H^2(X,\mathbb{Z})$. For any $Y$-sheaf $F$ we have
    $$
    e^{\frac{\xi'}{r}} \ch_G(F) = (r,\xi',\tfrac{1}{2} \xi^{\prime 2} - c_2')
    $$
    if and only if
    $$
    e^{\frac{\xi}{r}} \ch_G(F) = (r,\xi,\tfrac{1}{2} \xi^{2} - c_2), \quad c_2' =c_2 + (r-1) \gamma \xi +\tfrac{1}{2} r (r-1) \gamma^2.
    $$
    In particular, $M_{Y,\xi'/r}^H(r,\xi',c_2') = M_{Y,\xi/r}^H(r,\xi,c_2)$ and $\mathrm{vd}(r,\xi',c_2') = \mathrm{vd}(r,\xi,c_2)$. Summing over all $c_2$, $c_2'$, the result follows.
\end{proof}

 The independence of choice of $Y$ is perhaps not surprising, as this is just a particular model for the theory of twisted sheaves that we are using (though it is non-trivial that the Chern classes behave well under the equivalence discussed in Proposition \ref{prop:indep}). The independence of the choice of $\xi$ is also interesting: here we really need that we are dealing with generating functions and sum over all $c_2$. A different choice leads to a permutation of the terms of the generating function.
  
We make some observations on the generating functions in Definitions \ref{def:GLgenfun}, \ref{def:PGLgenfun}:
\begin{itemize}
\item The $\mathrm{SL}_r$ generating function is invariant under replacing $\mathcal{E}$ by $\mathcal{E} \otimes \mathcal{L}$ for any line bundle $\mathcal{L}$ on $X \times M$.  Let $L$ be a line bundle on $X$ and suppose $- \otimes L$ preserves stability, i.e., it induces an isomorphism
 \[
 \bigsqcup_{c_2} M_{X}^H(r,c_1,c_2) \to  \bigsqcup_{c_2} M_{X}^H(r,c_1+rc_1(L),c_2). 
 \]
 This happens, for example, when (1) $L$ is a multiple of $H$, or (2) when $\gcd(r,c_1H) = 1$ so that Gieseker stability and $\mu$-stability coincide \cite[Lem.~1.2.13, 1.2.14]{HL}. Then
 \[
 \mathsf{Z}_{(X,H),c_1}^{\mathrm{SL}_r, \mathsf{P}}(q) = \mathsf{Z}_{(X,H),c_1+rc_1(L)}^{\mathrm{SL}_r, \mathsf{P}}(q).
 \]
\item Suppose $\alpha(Y) \in \mathrm{Br}(X)$ has order $r>1$. Then stability is automatic by Remark \ref{rem:stabauto} and we write
\[
 \mathsf{Z}_{X,w}^{\mathrm{PGL}_r, \mathsf{P}}(q) = \mathsf{Z}_{(X,H),w}^{\mathrm{PGL}_r, \mathsf{P}}(q).
 \]
 \item Many interesting virtual invariants can be obtained from appropriate choices of formal insertions $\mathsf{P}$ by the virtual Hirzebruch-Riemann-Roch theorem \cite{CFK, FG}:
 \begin{itemize}
 \item Virtual Euler characteristic $e^{\mathrm{vir}}$, Hirzebruch genus $\chi_{-y}^{\mathrm{vir}}$, elliptic genus $Ell^{\mathrm{vir}}$. These were defined by Fantechi-G\"ottsche \cite{FG} and studied in the $\mathrm{SL}_r$ case in, e.g., \cite{GK1, GK3, GKL}. 
 \item Donaldson invariants and $K$-theoretic Donaldson invariants studied in algebraic geometry by G\"ottsche-Nakajima-Yoshioka in the $\mathrm{SL}_r$ case in \cite{GNY1, GNY2, GNY3}.
 \item Virtual cobordism class defined by Shen \cite{She} and studied in the $\mathrm{SL}_r$ case in \cite{GK2}.
 \item Virtual Segre and Verlinde numbers \cite{GKW, GK4}.
 \item Virasoro operators \cite{vBr, BLM}.
 \end{itemize}
\end{itemize}

We now make a comparison between moduli of sheaves on a \emph{trivial} Brauer-Severi variety $\mathbb{P}(E^\vee) \to X$ and moduli of sheaves on $X$.
\begin{proposition} \label{prop:untwist}
Let $(X,H)$ be a smooth polarized surface satisfying $H_1(X,\mathbb{Z}) = 0$ and consider a projective bundle $\pi : Y = \mathbb{P}(E^\vee) \to X$ for a rank $r$ vector bundle $E$ on $X$. Let $w := w(Y) \in H^2(X,\mu_r)$ and let $c_1 \in H^2(X,\mathbb{Z})$ be a (necessarily algebraic) class representing $w$. Suppose $\gcd(r,c_1 H) = 1$. Then there exists an isomorphism between Yoshioka and Gieseker-Maruyama-Simpson moduli spaces
\[
\phi : M_{Y,c_1/r}^H(r,c_1,c_2) \cong M_{X}^H(r,c_1,c_2).
\]
Moreover, under this isomorphism we have the following comparison of virtual fundamental classes and (normalized) universal sheaves\footnote{Recall that the universal sheaves $\mathcal{E}_1$, $\mathcal{E}_2$ may only exist \'etale locally, but their normalized versions exists globally by Remark \ref{rem:dropass}.}
\begin{align*}
\phi^* T_{M_{X}^H(r,c_1,c_2)}^{\mathrm{vir}} \cong T_{M_{Y,c_1/r}^H(r,c_1,c_2)}^{\mathrm{vir}}, \quad \phi^* [M_{X}^H(r,c_1,c_2)]^{\mathrm{vir}} &= [M_{Y,c_1/r}^H(r,c_1,c_2)]^{\mathrm{vir}}, \\
\pi^* \phi^* \Big(  (\mathcal{E}_2)^{\otimes r} \otimes \det(\mathcal{E}_2)^{-1} \Big) \cong \mathcal{E}_1^{\otimes r} \otimes \det(\mathcal{E}_1)^{-1},
\end{align*}
where $\pi$ and $\phi$ denote the appropriately base-changed morphisms. 
\end{proposition}
\begin{proof}
There exists an equivalence of categories \cite[Lem.~1.7]{Yoshioka2006}
\[
\Xi : \mathrm{Coh}(X,Y) \to \mathrm{Coh}(X), \quad F \mapsto \pi_*(F(-1)).
\]
We fix the $B$-field $\xi:=c_1$. Using $G = \pi^*E(1)$ and the projection formula, a direct calculation shows that for any $Y$-sheaf $F$, we have
\[
\mathrm{ch}_G^\xi(F) = \mathrm{ch}(\Xi(F)).
\]
Hence $\rk(F) = \rk(\Xi(F))$. Suppose this rank is positive. Then, using Hirzebruch-Riemann-Roch, we find that the reduced Hilbert polynomials are related by
\begin{align*}
\frac{\chi(\pi_*(F \otimes G^\vee)(mH))}{a_2^G(F)} = &\frac{\chi(\Xi(F)(mH))}{a_2(\Xi(F))}  + \frac{c_1(E^\vee)H}{r H^2} m \\ 
&+ \frac{c_1(E^\vee)c_1(X)}{2rH^2} 
+ \frac{\mathrm{ch_2}(E^\vee)}{r H^2} + \frac{c_1(E^\vee)c_1(\Xi(F))}{r \rk(\Xi(F)) H^2}.
\end{align*}

We would like to prove the following claim: for any torsion free sheaf $F$ on $X$, we have 
\[
\frac{\chi(F'(mH))}{a_2(F')} + \frac{1}{r H^2} \frac{c_1(E^\vee)c_1(F')}{\rk(F')} < \frac{\chi(F(mH))}{a_2(F)} + \frac{1}{r H^2} \frac{c_1(E^\vee)c_1(F)}{\rk(F)},
\]
for any subsheaf $0 \neq F' \subsetneq F$, if and only if $F$ is $\mu$-stable \cite[Def.~1.2.12]{HL}. We may replace $\mathbb{P}(E^\vee)$ by $\mathbb{P}(E^\vee \otimes \mathcal{L}^N)$, where $\mathcal{L}$ is an ample line bundle on $X$ and $N \in \mathbb{Z}_{>0}$. This replaces $c_1(E^\vee)$ by $N r c_1(\mathcal{L}) + c_1(E^\vee)$ which is ample for $N \gg 0$. 
In other words, we may assume $E$ is chosen such that $c_1(E^\vee)$ is ample. In particular, for any torsion free sheaf $F$ on $X$ and subsheaf $F'$ with $\rk(F') = \rk(F)$, we have
\[
c_1(E^\vee)(c_1(F) - c_1(F')) \geq 0
\]
by the Nakai-Moishezon criterion. Assuming $\gcd(\rk(F),c_1(F) H) = 1$, the claim then easily follows.
We deduce that $\Xi$ induces an isomorphism of moduli spaces\footnote{This isomorphism is essentially established in \cite[Lem.~2.2]{Yoshioka2006}. Additionally, we show that in our context $E$-twisted stability on $X$ \cite[Def.~2.2]{Yoshioka2006} coincides with $\mu$-stability on $X$, and we fix Chern character twisted by $B$-field, which leads to a nice comparison of Chern classes.}
\[
\phi : M_{Y,c_1/r}^H(r,c_1,c_2) \stackrel{\cong}{\to} M_{X}^H(r,c_1,c_2).
\]

The second part of the proposition is established exactly as in the proof of Proposition \ref{prop:indep}.
\end{proof}

\subsection{Deformation invariance}

Let $(X,H)$ be a smooth polarized surface satisfying $H_1(X,\mathbb{Z}) = 0$. We pick a class $w \in H^2(X,\mu_r)$, a sequence of formal insertions $\mathsf{P}$, and we consider the generating function $\mathsf{Z}_{(X,H),w}^{\mathrm{PGL}_r, \mathsf{P}}(q)$. We will now show that this generating function is invariant under deformations of $(X,H)$.

Let $f : \mathcal{X} \to B$ be a smooth projective morphism of relative dimension 2 with connected fibres over a smooth connected variety $B$. Suppose that one fibre (and hence all fibres) $\mathcal{X}_b$ satisfies $H_1(\mathcal{X}_b,\mathbb{Z}) = 0$. We fix a family of polarizations $\mathcal{H}$ on $\mathcal{X}$. We consider $\mu_r \subset \mathcal{O}_\mathcal{X}^{\times }$ as a constructible sheaf in the \'etale topology on $\mathcal{X}$. Then $R^2 f_* \mu_r$ is a constructible sheaf too and 
\[
(R^2 f_* \mu_r)_b \cong H^2(\mathcal{X}_b, \mu_r)
\]
for all closed points $b \in \mathcal{X}$ by the proper base change theorem \cite[Thm.~17.7]{MilneLec}. We fix a section 
\[
\widetilde{w} \in H^0(B,R^2 f_* \mu_r).
\]
Given a class $w \in H^2(\mathcal{X}_b, \mu_r)$ for some closed point $b \in B$, it is always possible to find such $\widetilde{w}$ in an \'etale neighbourhood of $b$ that restricts to $w$ over $b$. This fact uses the proper base change theorem.
We make the following two assumptions:
\begin{itemize}
\item $r$ is prime, and;
\item $\gcd(r, \widetilde{w}_b \mathcal{H}_b) = 1$ for some (and hence all) closed points $b \in B$.
\end{itemize}
This allows us to rule out strictly semistable objects as follows.
\begin{lemma} \label{lem:moduliprojinfam}
Suppose $r$ is prime and $\gcd(r, \widetilde{w}_b \mathcal{H}_b) = 1$ for some (and hence all) closed points $b \in B$. Then, for any closed point $b \in B$, any degree $r$ Brauer-Severi variety $Y \to \mathcal{X}_b$ with $w(Y) = \widetilde{w}_b$, and any $c_2$, the moduli space $M_{Y}^{\mathcal{H}_b}(r,0,c_2)$ is projective.
\end{lemma}
\begin{proof}
Since $r$ is prime, the order of the Brauer class $o(\widetilde{w}_{b}) \in \mathrm{Br}(\mathcal{X}_b)$ is either $r$ or $1$. In the former case, there are no strictly semistable objects because stability is automatic (Remark \ref{rem:stabauto}). In the latter case, we apply Proposition \ref{prop:untwist}.\footnote{Note that in Proposition \ref{prop:untwist}, we fix Chern character twisted by $B$-field, whereas in this remark we do not use the $B$-field. Concretely, for a (necessarily algebraic) lift $c_1$ of $\widetilde{w}_b$, we have $M_{Y}^{\mathcal{H}_b}(r,0,c_2) \cong M_{Y,c_1/r}^{\mathcal{H}_b}(r,c_1,c_2 + (r-1)c_1^2/2r)$.}
\end{proof}
We require the following result of de Jong used in his proof of the period-index theorem \cite{Jong2004} (and which does not require $r$ prime).\footnote{Strictly speaking, this result is not explicitly stated in loc.~cit.~but follows immediately from \cite[Prop.~3.2]{Jong2004} and the argument in \cite[Sect.~6]{Jong2004}.}
\begin{theorem}[de Jong] \label{thm:extending-azu-algs}
For any closed point $0 \in B$, there exists an \'etale neighbourhood $(U,0) \to (B,0)$ and a degree $r$ Azumaya algebra $\mathcal{A}$ on $\mathcal{X}_{U} := \mathcal{X} \times_B U$ such that
\[
w(\mathcal{A}_b) = \widetilde{w}_b \in H^2(\mathcal{X}_{U}|_b, \mu_r)
\]  
for all closed points $b \in U$.
\end{theorem}

The idea of the proof of this result is as follows. Pick a closed point $0 \in B$. By Theorem \ref{thm:gabber} (which is actually a consequence of the period-index theorem \cite{Lieblich2008}), there exists an Azumaya algebra $\mathcal{A}$ on $X$ such that $w(\mathcal{A}) = \widetilde{w}_0 \in H^2(\mathcal{X}_0,\mu_r)$. The obstruction space governing deformations of $\mathcal{A}$ is given by $H^2(\mathcal{X}_0,\mathcal{A} / \mathcal{O}_{\mathcal{X}_0})$. de Jong shows that (after applying an elementary transformation to $\mathcal{A}$), we may choose $\mathcal{A}$ such that $H^2(\mathcal{X}_0,\mathcal{A} / \mathcal{O}_{\mathcal{X}_0}) = 0$ \cite[Prop.~3.2]{Jong2004}. Using Grothendieck's existence theorem and Artin's approximation theorem, it then follows that there exists an \'etale neighbourhood $U \to B$ and a degree $r$ Azumaya algebra $\mathcal{A}$ on $\mathcal{X}_{U} := \mathcal{X} \times_B U$ such that $w(\mathcal{A}_0) = \widetilde{w}_0$. Using the proper base change theorem, we can go to a further \'etale neighbourhood where $w(\mathcal{A}_b) = \widetilde{w}_b$ for all closed points $b \in U$.

For the following theorem, we fix $\boldsymbol{a} = (a_1, \ldots, a_N) \in \{1,2\}^N$ and a family of formal insertions $\mathsf{P} = \{\mathsf{P}_{v,\boldsymbol{a}}\}_{v \geq 0}$ (Section \ref{sec:genfun}). Furthermore, for each $i = 1, \ldots, N$, we fix a flat section
\[
\widetilde{\alpha}_i \in H^0(B,\mathcal{H}^{2a_i}_{\mathrm{dR}}).
\]
We suppose $r$ is prime and $\gcd(r, \widetilde{w}_b \mathcal{H}_b) = 1$ for some (and hence all) closed points $b \in B$. By Lemma \ref{lem:moduliprojinfam}, the $\mathrm{PGL}_r$ generating function $\mathsf{Z}_{(\mathcal{X}_{b},\mathcal{H}_{b}),\widetilde{w}_{b}}^{\mathrm{PGL}_r, \mathsf{P}}(q)$ is defined on all the fibres of $\mathcal{X} \to B$. 
\begin{theorem} \label{thm:definv}
Let $\widetilde{w} \in H^0(B,R^2 f_* \mu_r)$ be a section. Suppose $r$ is prime and $\mathrm{gcd}(r,\widetilde{w}_b \mathcal{H}_b)=1$ for some (and hence all) closed points $b \in B$. Then $\mathsf{Z}_{(\mathcal{X}_{b},\mathcal{H}_{b}),\widetilde{w}_{b}}^{\mathrm{PGL}_r, \mathsf{P}}(q)$ is independent of the closed point $b \in B$.
\end{theorem}
\begin{proof}
First, fix a closed point $b \in \mathcal{B}$ and let $X := \mathcal{X}_b$, $H:=\mathcal{H}_b$, $w := \widetilde{w}_b$, and $\alpha_i := (\widetilde{\alpha}_i)_b$ for all $i$.  As discussed in Remark \ref{rem:notfixBfield}, it is not desirable to extend the $B$-field in families, so we first observe
\[
\mathsf{Z}_{(X,H),w}^{\mathrm{PGL}_r, \mathsf{P}}(q) = \sum_{c_2 \in \mathbb{Q}} q^{\frac{\mathrm{vd}(r,0,c_2)}{2r}} \int_{[M_{Y}^H(r,0,c_2)]^{\mathrm{vir}}} \mathsf{P}_{\mathrm{vd}(r,0,c_2),\boldsymbol{a}}(r,X,\boldsymbol{\alpha};M_{Y}^H(r,0,c_2),\mathcal{E}).
\]
We focus on deformation invariance of the term indexed by $c_2 \in \mathbb{Q}$ in the generating function. Note that $\ch = (r,0,-c_2)$ trivially extends to a flat section $\widetilde{v} \in  \bigoplus_{p} \Gamma(B, \mathcal{H}^{2p}_{\mathrm{dR}})$. 

By Theorem \ref{thm:extending-azu-algs}, we can cover the base $B$ by (connected) \'etale opens $U \to B$, such that on each $\mathcal{X}_U \to U$, we have an Azumaya algebras $\mathcal{A}$ of degree $r$ with $w(\mathcal{A}_b) = \widetilde{w}_b$ for all $b \in U$. Fix such an open. By the correspondence of  Proposition \ref{prp:azu-bs-equiv}, we therefore have a family of degree $r$ Brauer-Severi varieties 
\[
\mathcal{Y}_U \to \mathcal{X}_U
\]
such that $w(\mathcal{Y}_U|_b) = \widetilde{w}_b$ for all $b \in U$. Hence, we have a relative moduli space $M_{\mathcal{Y}_U/\mathcal{X}_U}$ and virtual fundamental class $[M_{\mathcal{Y}_U/\mathcal{X}_U}]^{\mathrm{vir}}$ as in Section \ref{sec:obthy}. Since $\mathsf{P}$ lifts to an insertion on $M_{\mathcal{Y}_U/\mathcal{X}_U}$, we deduce from \eqref{eqn:Gysinpull} that the term corresponding to $c_2$ of $\mathsf{Z}_{(\mathcal{X}_{b},\mathcal{H}_{b}),\widetilde{w}_{b}}^{\mathrm{PGL}_r, \mathsf{P}}(q)$ is independent of $b$ in the image of $U \to B$.
\end{proof}

\section{$\mathrm{PGL}_r$--$\mathrm{SL}_r$ correspondence} \label{sec:correspondence}

\subsection{Main result} \label{sec:mainresult}

Let $f : \mathcal{X} \to B$ be a smooth projective morphism of relative dimension $n$ with connected fibres over a smooth connected variety $B$. We first let the relative dimension $n$ be arbitrary. Suppose $H_1(\mathcal{X}_b,\mathbb{Z})_{\mathrm{tor}} = 0$ for some (and hence all) closed points $b \in B$. We recall that the stalk of the constructible sheaf $R^2 f_* \mu_r$ at $b$ is
\[
(R^2 f_* \mu_r)_b \cong H^2(\mathcal{X}_b, \mu_r)
\]
by the proper base change theorem. We fix a section 
\[
\widetilde{w} \in H^0(B,R^2 f_* \mu_r).
\]

We will now work in the complex analytic topology and make the following assumption: there exists a closed point $0 \in B$ and a class $\beta \in H^{1,1}(\mathcal{X}_0)$ such that the composition
\begin{equation}\label{eq:local-criterion}
    T_{B,0} \to H^1(\mathcal{X}_0, T_{\mathcal{X}_0}) \to H^2(\mathcal{X}_0, T_{\mathcal{X}_0} \otimes \Omega^1_{\mathcal{X}_0}) \to H^2(\mathcal{X}_0, \mathcal{O}_{\mathcal{X}_0})
\end{equation}
is surjective. Here the first map is the Kodaira-Spencer map, the second map is $\cup \beta$ and the final map is contraction. 

By \cite{VoiHL, VoiBookII}, this map can also be expressed in terms of the Gau\ss-Manin connection 
\[
\nabla : \mathcal{H}^2_{\mathrm{dR}} \to \mathcal{H}^2_{\mathrm{dR}} \otimes \Omega_B^1
\]
 as $\overline{\nabla}_0(\beta)$, where $\beta$ is viewed as an element of $(F^1 \mathcal{H}^2_{\mathrm{dR}} / F^2 \mathcal{H}^2_{\mathrm{dR}})|_0$. The surjectivity of \eqref{eq:local-criterion} implies that the Noether-Lefschetz locus of $\beta$, i.e.~locally analytically the locus in $B$ where $\beta$ remains of Hodge type $(1,1)$, is smooth of ``expected'' codimension $h^{2,0}(\mathcal{X}_0)$ at $0$ \cite{VoiHL}. Precisely the same assumption is made by Green to show that the Hodge locus is dense \cite{Gre}. When $n=2$ and $\beta$ is Poincar\'e dual to an algebraic class on $\mathcal{X}_0$, the same assumption is also used in the theory of reduced virtual fundamental classes in \cite{KT1, KT2}. Note that $\beta$ does not have to be related to $\widetilde{w}$ --- all we require is the existence of \emph{some} class with Noether-Lefschetz locus of expected codimension. 	

\begin{example}
    The family $\mathcal{X} \to B \subset |\mathcal{O}_{\mathbb{P}^3}(d)|$ of smooth degree $d$ surfaces in $\mathbb{P}^3$ of degree $d \geq 4$ has the property that the locus of points in $B$ admitting $\beta$, such that \eqref{eq:local-criterion} is surjective, is dense by \cite{Kim_1991} (see also \cite{CHM}). 
\end{example}
 
\begin{example}
    Let $U$ be the moduli space of relatively minimal elliptic surfaces $X \to \mathbb{P}^1$ with a section, irregularity $q = 0$, and geometric genus $p_g \geq 2$. Suppose that $X$ is an elliptic surface which has no reducible fibres and $j$-invariant not identically zero. Then \cite[Prop.~1.18]{Cox_1990} calculates the rank of \eqref{eq:local-criterion}. Using this description, one can find classes $\beta$ satisfying \eqref{eq:local-criterion}.
\end{example}

For further examples, we also refer to \cite{BGL}.

\begin{proposition} \label{thm:deform-w-trivial-new}
   Suppose \eqref{eq:local-criterion} is surjective. For any complex analytic simply connected neighbourhood $U$ of $0 \in B$, there exists a point $b \in U$ such that the Brauer class of $\widetilde{w}_{b}$ is trivial.
\end{proposition}
\begin{proof}
As in \cite[Prop.~5.20]{VoiBookII}, we first observe that \eqref{eq:local-criterion} is a Zariski open condition on $\beta$, so we may take $\beta \in H^{1,1}(\mathcal{X}_0) \cap H^2(\mathcal{X}_0, \mathbb{R})$. Take any simply connected open neighbourhood $U$ of $0$. Using the notation of loc.~cit., we have a trivialization
\[
t : \mathcal{H}_{\mathbb{R}}^2 \cong U \times H^2(\mathcal{X}_0,\mathbb{R})
\]
and (on $U$) we consider the composition
\[
\phi : \mathcal{H}_{\mathbb{R}}^{1,1} \hookrightarrow \mathcal{H}_{\mathbb{R}}^2 \stackrel{t}{\longrightarrow} U \times H^2(\mathcal{X}_0,\mathbb{R})  \stackrel{\pi}{\longrightarrow} H^2(\mathcal{X}_0,\mathbb{R}),
\]
where $\pi$ is the projection. Surjectivity of \eqref{eq:local-criterion} implies that this map is a submersion at $(0,\beta)$ by \cite[Lem.~5.22]{VoiBookII}, and hence a submersion in a neighbourhood $W$ of $(0,\beta)$.
 In particular $\phi$ is open on $W$. We note that $t$ preserves rational and integral classes. Furthermore, viewing $R^2 f_* \mu_r$ as a local system in the complex analytic topology,\footnote{Note that the image of the constructible \'etale sheaf $R^2f_*\mu_r$ under the map $\Sh(B_{\text{\'et}}) \to \Sh(B_{\text{an}})$ is the analytic sheaf $R^2f_*\mu_r$ by the comparison theorem \cite{MilneLec}.} $t$ trivializes $R^2f_* \mu_r|_U$ to the constant sheaf $H^2(\mathcal{X}_0,\mu_r)$.  

Still working over $U$, since $H^2(\mathcal{X}_0,\mathbb{Q})$ is dense in $H^2(\mathcal{X}_0,\mathbb{R})$ we can pick a $(b',\gamma) \in W$ with $\pi(t(b',\gamma))$ rational, and hence $\gamma$ rational.
Now choose a representative $\xi \in H^2(\mathcal{X}_0,\mathbb{Z})$ of $\widetilde{w}_0$ (Lemma \ref{lmm:rep-w-xi}). Consider the class
\[
\pi(t(b',\gamma)) + \frac{\xi}{r N}
\] 
for any $N \in \mathbb{Z}_{>0}$. Using the fact that $\phi$ is open in a neighbourhood of $(b',\gamma)$, we can choose $N \gg 0$ such that $\pi(t(b',\gamma)) + \xi / rN$ lies in the image of $\phi$ and such that $N \pi(t(b',\gamma))$ is integral. Then there exists a $(b,\delta) \in \mathcal{H}_{\mathbb{R}}^{1,1}$ mapping to $\pi(t(b',\gamma)) + \xi / rN$. Since $t$ preserves rational and integral classes, we deduce that $\epsilon:=rN \delta$ is an integral $(1,1)$ class on $\mathcal{X}_{b}$. Moreover, $\epsilon \mod r H^2(\mathcal{X}_{b},\mathbb{Z})$ corresponds to $\xi \mod r H^2(\mathcal{X}_{0},\mathbb{Z})$. Hence $\epsilon$ represents $\widetilde{w}_{b}$, and $\widetilde{w}_{b}$ has trivial Brauer class by the Kummer sequence \eqref{eqn:Kummer}.
\end{proof}

We now come to the main theorem of the paper. We fix the relative dimension of $f : \mathcal{X} \to B$ to be 2. We also choose a relative very ample divisor $\mathcal{H}$ with respect to $f$. Over some closed point $0 \in B$, we fix a Stiefel-Whitney class
\[
w \in H^2(\mathcal{X}_0,\mu_r).
\]
We fix $\boldsymbol{a} = (a_1, \ldots, a_N) \in \{1,2\}^N$ and a family of formal insertions $\mathsf{P} = \{\mathsf{P}_{v,\boldsymbol{a}}\}_{v \geq 0}$ (Section \ref{sec:genfun}). Furthermore, for each $i = 1, \ldots, N$, we fix an algebraic class
\[
\alpha_i \in H^{2a_i}(\mathcal{X}_0,\mathbb{Q}).
\]
By the proper base change theorem, after replacing $B$ by an \'etale neighbourhood of $0$, we may assume that $w$ extends to a section
\[
\widetilde{w} \in H^0(B,R^2 f_* \mu_r).
\]
By base change for Hodge bundles, after replacing $B$ by a Zariski neighbourhood of $0$, we may assume the classes $\alpha_i$ extend to flat sections
\[
\widetilde{\alpha}_i \in H^0(B,\mathcal{H}^{2a_i}_{\mathrm{dR}}).
\]

\begin{theorem} \label{thm:PGL-SLcorr}
Let $\widetilde{w} \in H^0(B,R^2 f_* \mu_r)$ be a section. Suppose $r$ is prime and $\mathrm{gcd}(r,\widetilde{w}_b \mathcal{H}_b)=1$ for some (and hence all) closed points $b \in B$. Suppose for some closed point $0 \in B$, there exists a class $\beta \in H^{1,1}(\mathcal{X}_0)$ such that the following composition is surjective
\[
T_B|_{0} \stackrel{\mathrm{KS}_0}{\longrightarrow} H^1(\mathcal{X}_0, T_{\mathcal{X}_0}) \stackrel{\cup \beta}{\longrightarrow} H^{2}(\mathcal{X}_0,\mathcal{O}_{\mathcal{X}_0}),
\]
where the first arrow is the Kodaira-Spencer map and the second is cupping with $\beta$ followed by contraction. Then any complex analytic simply connected neighbourhood $U$ of $0$ contains a closed point $b \in U$ such that $\widetilde{w}_{b} \in H^2(\mathcal{X}_{b},\mu_r)$ has trivial Brauer class and
\[
\mathsf{Z}_{(\mathcal{X}_{0},\mathcal{H}_{0}),\widetilde{w}_{0}}^{\mathrm{PGL}_r, \mathsf{P}}(q) = \mathsf{Z}_{(\mathcal{X}_{b},\mathcal{H}_{b}),c_1}^{\mathrm{SL}_r, \mathsf{P}}(q),
\]
where $c_1 \in H^2(\mathcal{X}_{b},\mathbb{Z})$ is any (necessarily algebraic) representative of $\widetilde{w}_{b}$.
\end{theorem}
\begin{proof}
Let $U$ be any complex analytic simply connected neighbourhood of $0$. By Proposition \ref{thm:deform-w-trivial-new}, there exists a closed point $b \in U$ such that $\widetilde{w}_{b}$ has trivial Brauer class. By Theorem \ref{thm:definv}, we have
\[
\mathsf{Z}_{(\mathcal{X}_{0},\mathcal{H}_{0}),\widetilde{w}_{0}}^{\mathrm{PGL}_r, \mathsf{P}}(q) = 
\mathsf{Z}_{(\mathcal{X}_{b},\mathcal{H}_{b}),\widetilde{w}_{b}}^{\mathrm{PGL}_r, \mathsf{P}}(q). 
\]
Pick a (necessarily algebraic) representative $c_1 \in H^2(\mathcal{X}_{b},\mathbb{Z})$ of $\widetilde{w}_{b}$ (Lemma \ref{lmm:rep-w-xi}, \eqref{eqn:Kummer}). 
Pick a degree $r$ Brauer-Severi variety over $\mathcal{X}_b$ with Stiefel-Whitney class $\widetilde{w}_{b}$ (Theorem \ref{thm:period-index}), which is therefore of the form $\mathbb{P}(E^\vee)$ for a rank $r$ vector bundle $E$ on $\mathcal{X}_b$ (Section \ref{sec:models}).
By Proposition \ref{prop:untwist}, we have
\[
\mathsf{Z}_{(\mathcal{X}_{b},\mathcal{H}_{b}),\widetilde{w}_{b}}^{\mathrm{PGL}_r, \mathsf{P}}(q) = \mathsf{Z}_{(\mathcal{X}_{b},\mathcal{H}_{b}),c_1}^{\mathrm{SL}_r, \mathsf{P}}(q) 
\]
which establishes the result.
\end{proof}

\subsection{Application to Vafa-Witten theory}

We recall the ``horizontal'' universality conjecture of the fourth-named author and G\"ottsche-Laarakker \cite[Conj.~1.10]{GKL}. We defined $\overline{\Delta}(q)$ and $\epsilon_r$ in the introduction.
\begin{conjecture}[G\"ottsche-Kool-Laarakker] \label{conj:GKL}
For any $r>1$, there exist\footnote{These universal functions only depend on $r$. Note that we use a different normalization of $D_0$ in loc.~cit.} 
\[
D_0, \{D_{ij}\}_{1 \leq i \leq j \leq r-1} \in \mathbb{C}[\![q^{\frac{1}{2r}}]\!] 
\]
with the following property. 
For any smooth polarized surface $(X,H)$ satisfying $b_1(X) = 0$, $h^{2,0}(X)>0$, $c_1 \in  H^2(X,\mathbb{Z})$, and $c_2 \in H^4(X,\mathbb{Z})$ such that there are no rank $r$ strictly Gieseker $H$-semistable sheaves on $X$ with Chern classes $c_1,c_2$, the virtual Euler characteristic $e^{\mathrm{vir}}(M_X^H(r,c_1,c_2))$ equals the coefficient of $q^{\mathrm{vd}(r,c_1,c_2)/2r}$ in
\begin{align*}
&r^{2+K_X^2 - \chi(\mathcal{O}_X)} \Bigg( \frac{1}{\overline{\Delta}(q^{\frac{1}{r}})^{\frac{1}{2}}} \Bigg)^{\chi(\mathcal{O}_X)} D_0^{K_X^2} \sum_{(a_1, \ldots, a_{r-1}) \in H^2(X,\mathbb{Z})^{r-1}}   \prod_{i} \epsilon_r^{i a_i c_1} \, \SW(a_i) \prod_{i \leq j} D_{ij}^{a_i a_j}.
\end{align*}
\end{conjecture}

Evidence for this conjecture was obtained by direct implementation of Mochizuki's formula \cite{Moc} in \cite{GK1, GK3}. We deduce the following:
\begin{corollary}
Suppose $X = \mathcal{X}_0$ and $w = \widetilde{w}_0$ for a family $\mathcal{X} \to B$ satisfying the conditions of Theorem \ref{thm:PGL-SLcorr} and $h^{2,0}(X)>0$. Fix any $\delta \in \mathbb{Z}$ such that $\delta \equiv -(r-1) w^2 - (r^2-1) \chi(\mathcal{O}_X) \mod 2r$. Then Conjecture \ref{conj:GKL} implies that the coefficient of $q^{\delta/2r}$ in  $\mathsf{Z}_{(X,H),w}^{\mathrm{PGL}_r, \mathsf{Eu}}(q)$ equals the coefficient of $q^{\delta/2r}$ in
\begin{align*}
&r^{2+K_X^2 - \chi(\mathcal{O}_X)} \Bigg( \frac{1}{\overline{\Delta}(q^{\frac{1}{r}})^{\frac{1}{2}}} \Bigg)^{\chi(\mathcal{O}_X)} D_0^{K_X^2} \sum_{(a_1, \ldots, a_{r-1}) \in H^2(X,\mathbb{Z})^{r-1}}   \prod_{i} \epsilon_r^{i a_i w} \, \SW(a_i) \prod_{i \leq j} D_{ij}^{a_i a_j}.
\end{align*}
\end{corollary}

Closed conjectural expressions for $D_0, D_{ij}$ can be found in \cite{GK1,GK3,GKL} for $r=2,3,5$. For $r=2,3$ they are expressed in terms of the Dedekind eta function and the theta functions of the $A_{r-1}^\vee$ lattice \cite{GK1, GK2}. For $r=5$, the expressions also involve the Rogers-Ramanujan continued fraction \cite{GKL}.

Let $X$ be a smooth projective surface satisfying $H_1(X,\mathbb{Z}) = 0$, $h^{2,0}(X)>0$, and let $c_1 \in H^2(X,\mathbb{Z})$ be an algebraic class. In \cite{TT}, Tanaka-Thomas give a mathematical definition of the $\mathrm{SU}(r)$ Vafa-Witten partition function $\mathsf{VW}^{\mathrm{SU}(r)}_{X,c_1}(q)$. As mentioned in the introduction, for $r$ prime, a definition of the $\mathrm{PSU}(r)$ Vafa-Witten partition function was given in \cite{JiangKool2021} (see also \cite{Jiang}). It is of the form
\begin{align*}
&\mathsf{VW}^{\mathrm{PSU}(r)}_{X,c_1}(q) = \sum_{w \in H^2(X,\mu_r)} \epsilon_r^{c_1 w} \, \mathsf{VW}_{X,w}(q).
\end{align*}
For $o(w) = 0$, one is reduced to untwisted Higgs pairs and $\mathsf{VW}_{X,w}(q)$ can be defined using Tanaka-Thomas's approach. For $o(w) \neq 0$, we take \cite{JiangKool2021}
\[
\mathsf{VW}_{X,w}(q) := \mathsf{Z}^{\mathrm{PGL}_r,\mathsf{Eu}}_{X,w}(q).
\]

Assuming the above-mentioned closed conjectural expressions for the universal functions for $r=2,3,5$, the following $S$-duality conjecture (due to Vafa-Witten \cite{VW} and mathematically formulated in \cite{JiangKool2021}) was checked for $r=2$ \cite{VW, DPS, GK3}, $r=3$ \cite{GK3}, and $r=5$ \cite{GKL}: 
\begin{conjecture}[Vafa-Witten] \label{intro:Sdualconj} 
Let $(X,H)$ be a smooth polarized  surface satisfying $H_1(X,\mathbb{Z}) = 0$ and $h^{2,0}(X)>0$. Let $r$ be prime and $c_1 \in H^2(X,\mathbb{Z})$ algebraic. Then $\mathsf{VW}_{X,c_1}^{\mathrm{SU}(r)}(q)$ and $\mathsf{VW}_{X,c_1}^{\mathrm{PSU}(r)}(q)$ are Fourier expansions in $q = \exp(2 \pi \sqrt{-1} \tau)$ of meromorphic functions $\mathsf{VW}_{X,c_1}^{\mathrm{SU}(r)}(\tau)$ and $\mathsf{VW}_{X,c_1}^{\mathrm{PSU}(r)}(\tau)$ on the upper half plane satisfying
\begin{equation*} 
\mathsf{VW}^{\mathrm{SU}(r)}_{X,c_1}(-1/\tau) = (-1)^{(r-1)\chi(\mathcal{O}_X)} \Big( \frac{ r \tau}{\sqrt{-1}} \Big)^{-\frac{e(X)}{2}} \mathsf{VW}^{\mathrm{PSU}(r)}_{X,c_1}(\tau).
\end{equation*}
\end{conjecture}

\subsection{Application to \texorpdfstring{$c_2^{\mathrm{min}}$}{c2min} of Azumaya algebras} \label{sec:minc2}

Let $X$ be a smooth projective surface with $H_1(X,\mathbb{Z}) = 0$ and function field $\mathbb{C}(X)$. Let $D$ be a (central) division algebra over $\mathbb{C}(X)$ of degree $r > 1$ (equivalently, an element of $\mathrm{Br}(\mathbb{C}(X))$ of order $r$). We assume $D$ lies in the image of $\mathrm{Br}(X) \hookrightarrow \mathrm{Br}(\mathbb{C}(X))$. As mentioned in the introduction, we are interested in the smallest $c_2$ for which there exists an Azumaya algebra $\mathcal{A}$ on $X$ whose stalk over the generic point $\eta \in X$ is isomorphic to $D$. We refer to this value as $c_2^{\mathrm{min}}$ and recall that Artin-de Jong proved \cite[Cor.~7.1.5, Thm.~7.2.1]{AdJ}
\[
c_2^{\mathrm{min}} \geq \mathrm{max}\big\{r^2 \chi(\mathcal{O}_X) - h^0(\omega_X^{\otimes r}) -1 , 0 \big\}.
\]
Our general strategy is as follows:
\begin{itemize}
\item Fix a class $w \in H^2(X,\mu_r)$ such that its Brauer class $o(w) \in \mathrm{Br}(X) \hookrightarrow \mathrm{Br}(\mathbb{C}(X))$ corresponds to $D$. This is possible since $H^2(X,\mu_r)$ surjects onto $H^2(X,\mathbb{G}_m)[r]$ by \eqref{eqn:Kummer}.
\item Fix any $\delta \in \mathbb{Z}$ such that $\delta \equiv -(r-1) w^2 - (r^2-1) \chi(\mathcal{O}_X) \mod 2r$. Suppose the coefficient of $q^{\delta/2r}$ of some $\mathrm{PGL}_r$ generating function $\mathsf{Z}_{(X,H),w}^{\mathrm{PGL}_r,\mathsf{P}}(q)$ is non-zero.
\item Denoting by $\xi \in H^2(X,\mathbb{Z})$ a lift of $w$ (Lemma \ref{lmm:rep-w-xi}), we conclude that for any degree $r$ Brauer-Severi variety $Y \to X$ with $w(Y) = w$ (which exist by Theorem \ref{thm:period-index}), the moduli space $M_{Y,\xi/r}^H(r,\xi,c_2)$ is non-empty, where $c_2 \in \mathbb{Z}$ is determined by the equation $\delta = 2rc_2 - (r-1) \xi^2 - (r^2-1) \chi(\mathcal{O}_X)$.
\item Therefore, by Corollary \ref{cor:HS}, there exists an Azumaya algebra $\mathcal{A}$ on $X$ whose stalk over the generic point is isomorphic to $D$ satisfying $c_2(\mathcal{A}) \leq \delta + (r^2-1) \chi(\mathcal{O}_X)$. In particular $c_2^{\mathrm{min}} \leq \delta + (r^2-1) \chi(\mathcal{O}_X)$.
\end{itemize}

We illustrate this strategy by focusing on the leading term of $\mathsf{Z}_{(X,H),w}^{\mathrm{PGL}_r,\mathsf{P}}(q)$. For this, we use the explicit form of the Mari\~{n}o-Moore conjecture \cite{MM, LM} due to G\"ottsche \cite{Gott}, which we now recall. For a smooth polarized surface $(X,H)$ satisfying $H_1(X,\mathbb{Z}) = 0$, we fix $r,c_1$ so that there are no rank $r$ strictly Gieseker $H$-semistable sheaves on $X$ with first Chern class $c_1$, and consider the Gieseker-Maruyama-Simpson moduli space $M_X^H(r,c_1,c_2)$. For any class $\alpha \in H^*(X,\mathbb{Q})$ and $k \in \mathbb{Z}_{\geq 0}$, one defines the $\mu$-insertion
\[
\mu(\alpha) :=  - \pi_{M*} \Big( \pi_X^* \alpha \cap \mathrm{ch}_2(\mathcal{E} \otimes \det(\mathcal{E})^{-\frac{1}{r}})\Big),
\]
where $\pi_M, \pi_X$ are the projections from $M \times X$ to $M,X$ respectively. Denote the Poincar\'e dual of the point class by $\mathrm{pt} \in H^4(X,\mathbb{Z})$.
We are interested in the following generating series of (algebro-geometric) $\mathrm{SL}_r$ Donaldson invariants
\[
\mathsf{Z}_{(X,H),c_1}^{\mathrm{SL}_r,\mathsf{D}}(z) = \sum_{c_2 \in \mathbb{Z}} z^{\mathrm{vd}(r,c_1,c_2)} \int_{[M_X^H(r,c_1,c_2)]^{\mathrm{vir}}} \exp\Big(\mu(L) + \mu(\mathrm{pt}) \cdot u \Big),
\]
where $L \in H^2(X,\mathbb{Z})$ and $u$ is a formal variable.
We first recall Witten's conjecture, proved in the algebro-geometric setup by G\"ottsche-Nakajima-Yoshioka \cite{GNY3}. 
\begin{theorem}[G\"ottsche-Nakajima-Yoshioka] \label{thm:GNY}
Let $(X,H)$ be a smooth polarized surface satisfying $H_1(X,\mathbb{Z}) = 0$ and $h^{2,0}(X)>0$. Let $c_1 \in H^2(X,\mathbb{Z})$ be such that there are no rank $2$ strictly Gieseker $H$-semistable sheaves on $X$ with first Chern class $c_1$. Then the coefficient of $z^{\mathrm{vd}(2,c_1,c_2)}$ in $\mathsf{Z}_{(X,H),c_1}^{\mathrm{SL}_2,\mathsf{D}}(z)$ equals the coefficient of $z^{\mathrm{vd}(2,c_1,c_2)}$ in
\[
2^{2-\chi(\mathcal{O}_X)+K_X^2} e^{( \frac{1}{2} L^2 + 2u) z^2} \!\! \sum_{a \in H^2(X,\mathbb{Z})} (-1)^{a c_1} \mathrm{SW}(a) e^{-(2a - K_X) L z }.
\]
\end{theorem}

The higher rank generalization of Witten's conjecture is known as the Mari\~{n}o-Moore conjecture \cite{MM} (see also \cite{LM}). We discuss its algebro-geometric version due to G\"ottsche \cite{Gott}. Define $[r-1] := \{1, \ldots, r-1\}$. We require the following numbers: for all $1 \leq i < j \leq r-1$
\begin{align*}
\beta_{ij} &:= 
\frac{\sin((i+j)\pi / 2r)}{\sin((j-i)\pi / 2r)} \in \mathbb{R}_{>0}, \quad \beta_{ji} := \beta_{ij}, \\
B &:= \sum_{I \subset [r-1]} \prod_{i \in I \atop j \in [r-1] \setminus I} \beta_{ij}.
\end{align*}

\begin{conjecture}[G\"ottsche] \label{conj:Gott}
Let $(X,H)$ be a smooth polarized surface satisfying $H_1(X,\mathbb{Z}) = 0$ and $h^{2,0}(X)>0$. Let $r>1$ and $c_1 \in H^2(X,\mathbb{Z})$ such that there are no rank $r$ strictly Gieseker $H$-semistable sheaves on $X$ with first Chern class $c_1$. Then the coefficient of $z^{\mathrm{vd}(r,c_1,c_2)}$ in $\mathsf{Z}_{(X,H),c_1}^{\mathrm{SL}_r,\mathsf{D}}(z)$ equals the coefficient of $z^{\mathrm{vd}(r,c_1,c_2)}$ in
\[
r^{2-\chi(\mathcal{O}_X)} B^{K_X^2} e^{( \frac{1}{2} L^2 + ru) z^2} \sum_{\boldsymbol{a}} \prod_{i} \epsilon_r^{i a_i c_1} \mathrm{SW}(a_i) e^{-\sin(i \pi / r) (2a_i - K_X) L z } \prod_{i<j} \beta_{ij}^{(2a_i-K_X)(a_j-a_i)},
\]
where the sum is over all $\boldsymbol{a} = (a_1, \ldots, a_{r-1}) \in H^2(X,\mathbb{Z})^{r-1}$.
\end{conjecture}

The $r=3,4,5$ cases of this conjecture also appear in \cite{GK4, GKL} as consequences of other conjectures.

Let $w \in H^2(X,\mu_r)$, $\pi : Y \to X$ a choice of degree $r$ Brauer-Severi variety with $w(Y)=w$ (Theorem \ref{thm:period-index}), and $\xi \in H^2(X,\mathbb{Z})$ a lift of $w$ (Lemma \ref{lmm:rep-w-xi}). We assume $r$ is prime and $\mathrm{gcd}(r,wH)=1$.
We consider the following generating series of  $\mathrm{PGL}_r$ Donaldson invariants
\[
\mathsf{Z}_{(X,H),w}^{\mathrm{PGL}_r,\mathsf{D}}(z) = \sum_{c_2 \in \mathbb{Z}} z^{\mathrm{vd}(r,\xi,c_2)} \int_{[M_{Y,\xi/r}^H(r,\xi,c_2)]^{\mathrm{vir}}} \exp\Big(\mu(L) + \mu(\mathrm{pt}) \cdot u \Big),
\]
where, for any $\alpha \in H^*(X,\mathbb{Q})$, we define
\[
\mu(\alpha) :=  - \pi_{M*} \Big( \pi_X^* \alpha \cap \mathrm{ch}_2(\pi_*(\mathcal{E} \otimes \det(\mathcal{E})^{-\frac{1}{r}}))\Big).
\]

\begin{corollary} \label{cor:conseqLGconj}
Suppose $X = \mathcal{X}_0$ and $w = \widetilde{w}_0$ for a family $\mathcal{X} \to B$ satisfying the conditions of Theorem \ref{thm:PGL-SLcorr} and $h^{2,0}(X)>0$. Fix any $\delta \in \mathbb{Z}$ such that $\delta \equiv -(r-1) w^2 - (r^2-1) \chi(\mathcal{O}_X) \mod 2r$. Then Conjecture \ref{conj:Gott} (which holds for $r=2$) implies that the coefficient of $z^{\delta}$ in $\mathsf{Z}_{(X,H),w}^{\mathrm{PGL}_r, \mathsf{D}}(z)$ equals the coefficient of $z^{\delta}$ in
\[
r^{2-\chi(\mathcal{O}_X)} B^{K_X^2} e^{( \frac{1}{2} L^2 + ru) z^2} \sum_{\boldsymbol{a}} \prod_{i} \epsilon_r^{i a_i w} \mathrm{SW}(a_i) e^{-\sin(i \pi / r) (2a_i - K_X) L z } \prod_{i<j} \beta_{ij}^{(2a_i-K_X)(a_j-a_i)},
\]
where the sum is over all $\boldsymbol{a} = (a_1, \ldots, a_{r-1}) \in H^2(X,\mathbb{Z})^{r-1}$.
\end{corollary}

Suppose $X$ is a minimal surface of general type satisfying $H_1(X,\mathbb{Z})=0$ and $h^{2,0}(X)>0$. Then the only Seiberg-Witten basic classes $a \in H^2(X,\mathbb{Z})$, i.e.~the only classes for which $\mathrm{SW}(a) \neq 0$, are $a=0,K_X$ and the corresponding Seiberg-Witten invariants are $1, (-1)^{\chi(\mathcal{O}_X)}$ \cite[Thm.~7.4.1]{Mor}. Then the formula of Corollary \ref{cor:conseqLGconj} simplifies dramatically. We take $u=0$ and $L=K_X$ and record the leading coefficients of the formula for $r=2,3$. The main observation is that in each case the leading coefficient is \emph{positive}.

\begin{example} \label{ex:rk2and3}
For $r=2$, we have the following cases:
\begin{itemize}
\item For $w K_X + \chi(\mathcal{O}_X) \equiv 0 \mod 2$, the leading term is $$2^{2 - \chi(\mathcal{O}_X) + K_X^2} \cdot 2 z^0.$$
Taking a representative $w = [\xi]$ with $\xi \in H^2(X,\mathbb{Z})$, in this case we have $\mathrm{vd}(2,w,c_2) \equiv 0 \mod 2$ for all $c_2 \in \mathbb{Z}$. Here we used Wu's formula $w^2 \equiv w K_X \mod 2$.
\item For $w K_X + \chi(\mathcal{O}_X) \equiv 1 \mod 2$, the leading term is $$2^{2 - \chi(\mathcal{O}_X) + K_X^2} (2K_X^2) z.$$
Taking a representative $w = [\xi]$ with $\xi \in H^2(X,\mathbb{Z})$, in this case we have $\mathrm{vd}(2,w,c_2) \equiv 1 \mod 2$ for all $c_2 \in \mathbb{Z}$.
\end{itemize}
\end{example}

\begin{example}
For $r=3$, we have the following cases: 
\begin{itemize}
\item For $w K_X \equiv 0 \mod 3$, $\chi(\mathcal{O}_X) \equiv 0 \mod 2$ the leading term is $$3^{2-\chi(\mathcal{O}_X) + K_X^2}(2^{1+K_X^2}+2) z^0.$$
\item For $w K_X \equiv 0 \mod 3$, $\chi(\mathcal{O}_X) \equiv 1 \mod 2$ the leading term is $$3^{2-\chi(\mathcal{O}_X) + K_X^2}(2^{1+K_X^2}-2) z^0.$$
\item For $w K_X \equiv 1,2 \mod 3$, $\chi(\mathcal{O}_X) \equiv 0 \mod 2$ the leading term is $$3^{2-\chi(\mathcal{O}_X) + K_X^2}(2^{1+K_X^2}-1) z^0.$$
\item For $w K_X \equiv 1,2 \mod 3$, $\chi(\mathcal{O}_X) \equiv 1 \mod 2$ the leading term is $$3^{2-\chi(\mathcal{O}_X) + K_X^2}(2^{1+K_X^2}+1) z^0.$$
\end{itemize}
Taking a representative $w = [\xi]$ with $\xi \in H^2(X,\mathbb{Z})$, in each of these cases we have $\mathrm{vd}(r,\xi,c_2) \equiv 0 \mod 2$ for any $c_2 \in \mathbb{Z}$.
\end{example}

\begin{remark} \label{rem:nonvanishconj}
It is natural to expect that the formula for $\mathsf{Z}_{(X,H),w}^{\mathrm{PGL}_r, \mathsf{D}}$ of Corollary \ref{cor:conseqLGconj} also holds when $r>1$ is not necessarily prime (at least when there are no strictly semistables). Experimentation for ranks $r>3$ leads us to the following expectations. For $X,H,w$ as above and any odd rank $r$, we have
\[
\mathsf{Z}_{(X,H),w}^{\mathrm{PGL}_r, \mathsf{D}}(0) \in \mathbb{Z}_{>0}.
\]
Moreover, for any even rank $r$, we have
\[
\left\{ 
\begin{array}{cc}
\mathsf{Z}_{(X,H),w}^{\mathrm{PGL}_r, \mathsf{D}}(0) \in \mathbb{Z}_{>0}  & \mathrm{if \ } wK_X + \chi(\mathcal{O}_X) \equiv 0 \mod 2 \\
\frac{\partial}{\partial z} \mathsf{Z}_{(X,H),w}^{\mathrm{PGL}_r, \mathsf{D}}(0) \in \mathbb{Z}_{>0} & \mathrm{if \ } wK_X + \chi(\mathcal{O}_X) \equiv 1 \mod 2.
\end{array}
\right.
\]
Taking a representative $w = [\xi]$ with $\xi \in H^2(X,\mathbb{Z})$, for $r$ odd we have $\mathrm{vd}(r,\xi,c_2) \equiv 0 \mod 2$, and for $r$ even we have $\mathrm{vd}(r,\xi,c_2) \equiv w K_X + \chi(\mathcal{O}_X) \mod 2$ for any $c_2 \in \mathbb{Z}$.
This positivity is obvious in the case $r$ is odd, $wK_X \equiv 0 \mod r$, and $\chi(\mathcal{O}_X) \equiv 0 \mod 2$, but appears non-trivial in general.\footnote{The expectation in this remark was first formulated for $r$ prime by the authors, and then generalized to any $r$ by G\"ottsche in an e-mail conversation.}
\end{remark}

Combining Example \ref{ex:rk2and3} and Corollary \ref{cor:conseqLGconj}, we deduce the following.

\begin{theorem}
Suppose $X = \mathcal{X}_0$ for a family $\mathcal{X} \to B$ satisfying the conditions of Theorem \ref{thm:PGL-SLcorr}, and $X$ is a minimal surface of general type satisfying $h^{2,0}(X)>0$. Let $D \in \mathrm{Br}(\mathbb{C}(X))$ be a degree $r$ division algebra in the image of $\mathrm{Br}(X) \hookrightarrow \mathrm{Br}(\mathbb{C}(X))$. Then, for $r=2$, we have
\[
c_2^{\mathrm{min}} \leq 3\chi(\mathcal{O}_X) + 1.
\]
Moreover, for $r=3$ and assuming G\"ottsche's conjecture \ref{conj:Gott} for $r=3$, we have
\[
c_2^{\mathrm{min}} \leq 8\chi(\mathcal{O}_X).
\]
\end{theorem}
\begin{proof}
Let $\alpha \in \mathrm{Br}(X)$ be the Brauer class corresponding to $D$. Let $Y$ be a degree $r$ Brauer-Severi variety with $o(w(Y)) = \alpha$ (Theorem \ref{thm:period-index}). 

For $r=2$, by Example \ref{ex:rk2and3}, there are two cases. (1) For $w K_X + \chi(\mathcal{O}_X) \equiv 0 \mod 2$ and running the strategy at the beginning of this section, we deduce $c_2^{\mathrm{min}} \leq 3\chi(\mathcal{O}_X)$. (2)  For $w K_X + \chi(\mathcal{O}_X) \equiv 1 \mod 2$, we deduce $c_2^{\mathrm{min}} \leq 3\chi(\mathcal{O}_X) + 1$.

For $r=3$, Example \ref{ex:rk2and3} implies the result.
\end{proof}

In general, if the expectation of Remark \ref{rem:nonvanishconj} holds and we take $X$ as in the previous theorem, then for any degree $r>1$ division algebra $D \in \mathrm{Br}(\mathbb{C}(X))$ in the image of $\mathrm{Br}(X) \hookrightarrow \mathrm{Br}(\mathbb{C}(X))$, we obtain 
\begin{align*}
c_2^{\mathrm{min}} &\leq (r^2-1)\chi(\mathcal{O}_X)+1, \quad \textrm{for } r \textrm{ even}, \\ 
c_2^{\mathrm{min}} &\leq (r^2-1)\chi(\mathcal{O}_X), \quad \textrm{for } r \textrm{ odd}. 
\end{align*}

\printbibliography

\noindent \begin{tabular}{ll}
{\tt{d.vanbree@uu.nl}} & {\tt{amingh@umd.edu}} \smallskip \\
Department of Mathematics & Department of Mathematics \\
Utrecht University & University of Maryland \\
PO Box 80010 & 4176 Campus Drive \\
3508 TA Utrecht & William E. Kirwan Hall \\
The Netherlands & College Park, MD 20742-4015 \\ 
& USA \\ \\
{\tt{y.jiang@ku.edu}} & {\tt{m.kool1@uu.nl}} \smallskip \\
Department of Mathematics & Department of Mathematics \\
University of Kansas & Utrecht University \\
405 Jayhawk Blvd & PO Box 80010 \\
Lawrence, KS 66045 & 3508 TA Utrecht \\
USA & The Netherlands \\

\end{tabular}
\end{document}